\documentclass[a4paper]{amsart}

\pdfoutput=1

\usepackage[a4paper,inner=3cm,outer=3cm,top=3.3cm,bottom=3cm]{geometry}
\usepackage[colorlinks=true,citecolor=blue, urlcolor=magenta, breaklinks=true]{hyperref}

\usepackage[english]{babel}
\usepackage{geometry}
\usepackage{amsfonts,amsmath,amssymb,amsthm,mathrsfs}
\usepackage{caption}
\usepackage[labelfont=rm]{subcaption}
\usepackage{xspace}
\usepackage{microtype}
\usepackage{enumerate}
\usepackage{booktabs}
\usepackage{float}

\usepackage{mathtools}
\usepackage{tikz}
\usetikzlibrary{arrows}
\usepackage{pgf}
\usepackage{pgfplots}
\usepackage{pgfplotstable}
\pgfplotsset{compat=1.13}

\microtypesetup{tracking,kerning,spacing}
\microtypecontext{spacing=nonfrench}

\newcommand\cprime{$'$}

\newcommand\polymake{{\tt polymake}\xspace}
\newcommand\topcom{{\tt TOPCOM}\xspace}

\DeclareSymbolFontAlphabet{\mathbb}{AMSb}
\newcommand\CC{{\mathbb C}}

\newcommand\QQ{{\mathbb Q}}
\newcommand\RR{{\mathbb R}}
\newcommand\ZZ{{\mathbb Z}}
\newcommand\HS{\mathbb{HS}}
\newcommand\PP{{\mathbb P}}

\DeclareMathOperator{\conv}{conv}
\DeclareMathOperator{\Newt}{Newt}
\DeclareMathOperator{\Vol}{Vol}

\newcommand{\st}{\,\colon}
\newcommand\cT{\mathcal{T}}
\newcommand\cP{\mathcal{P}}

\theoremstyle{plain}
    \newtheorem{theorem}{Theorem}
	\newtheorem{algorithm}[theorem]{Algorithm}
    \newtheorem{corollary}[theorem]{Corollary}
    \newtheorem{lemma}[theorem]{Lemma}
    \newtheorem{proposition}[theorem]{Proposition}
\theoremstyle{definition}
    
    \newtheorem{example}[theorem]{Example}
    \newtheorem{definition}[theorem]{Definition}

\author[G.\,Balletti]{Gabriele Balletti}
\address[G.\,Balletti]{Stockholm University\\Sweden}
\email{balletti@math.su.se}

\author[M.\,Panizzut]{Marta Panizzut}
\address[M.\,Panizzut]{TU Berlin\\Germany}
\email{panizzut@math.tu-berlin.de}

\author[B.\,Sturmfels]{Bernd Sturmfels}
\address[B.\,Sturmfels]{MPI Leipzig\\Germany and UC Berkeley\\USA}
\email{bernd@mis.mpg.de}


\title{K3 Polytopes and their Quartic Surfaces}

\begin{document}

\begin{abstract}
K3 polytopes appear in complements of  tropical quartic surfaces.
They are dual to regular unimodular central triangulations of reflexive polytopes in the fourth dilation of the standard tetrahedron.
Exploring these combinatorial objects, we classify 
K3 polytopes with up to $30$ vertices. Their number is $36\,297\,333$.
We study the singular loci of quartic surfaces that tropicalize to K3 polytopes.
These surfaces are stable in the sense of Geometric Invariant Theory.
\end{abstract}
\maketitle

\section{Introduction} 

Tropical hypersurfaces are defined by tropical polynomials. They support pure rational weighted polyhedral complexes.  
The regions in the complement of a tropical hypersurface are convex polyhedra. These are interesting for a
range of problems in geometric combinatorics. If the  tropical polynomial is a product of linear forms,
so the hypersurface is a hyperplane arrangement, then the bounded regions are \emph{polytropes} \cite{JK10}.
These are the basic building blocks in tropical convexity \cite{DS04}, and they arise in contexts ranging from
affine buildings \cite{JSY07} and Coxeter arrangements \cite{LP07}  to combinatorial optimization \cite{JL16}. 
The combinatorial types of polytropes were  classified by Tran \cite{Tran17}. 

The point of departure for this article is Exercise 13 in \cite[Section 1.9]{MS15}. It asks to show
that the unique bounded region in the complement of a smooth cubic curve in the tropical plane is 
an $m$-gon, where $m \in \{3,4,5,6,7,8,9\}$, and each of these seven possibilities occurs.
The boundary of this convex $m$-gon carries the group structure of the {\em tropical elliptic curve},
and its lattice length is the {\em tropical j-invariant}. These results are due to
Vigeland \cite{Vig09} and Katz-Markwig-Markwig \cite{KMM08}.

Elliptic curves are Calabi-Yau varieties. In higher dimensions, these varieties occupy 
a prominent place at the crossroads of algebraic geometry and theoretical physics.
Following Batyrev \cite{Bat94}, reflexive polytopes capture the combinatorial essence of
mirror symmetry for Calabi-Yau hypersurfaces.

The title of this paper refers to the bounded region of a smooth tropical quartic surface.  We call such a region a \emph{K3 polytope}. The name  is motivated by the fact that a smooth quartic surface in $\mathbb{P}^3$ is a {\em K3 surface},
that is, a non-singular surface with trivial canonical bundle and trivial first cohomology group. 
In short, our topic is the above Exercise 13, but now in one higher dimension.

The study of smooth tropical quartic surfaces and K3 polytopes is dual to the study of regular unimodular triangulations of a Newton polytope with one interior lattice point and contained in the scaled tetrahedron $4\Delta_3$. A naive approach to our problem is to compute the secondary fan of $4\Delta_3$ and then to filter out the unimodular triangulations.
 However, this  is not feasible with the current state of software and algorithms.
The established tools are \texttt{gfan} \cite{gfan} and \topcom \cite{topcom}. They use 
different  algorithms to pass through cones of the secondary fan: \texttt{gfan} computes a new weight by traversing a facet, while \topcom exploits bistellar flips. Jordan, Joswig and Kastner  \cite{JJK17}  introduced a new algorithm,
called \emph{down-flip reverse search}, for parallel enumeration of regular triangulations. 
Their implementation \texttt{mptopcom} generated results that are out of reach for 
 \texttt{gfan}  and \topcom. We refer to the summary in \cite[Table~3]{JJK17}. They also report that
the number of regular triangulations of $4\Delta_3$ appears to be ``out of reach for the current implementations, including \texttt{mptopcom}''.

Our first main result  is  a practical algorithm for classifying K3 polytopes. We use this to establish
\begin{theorem} The following result concerns tropical quartic surfaces that have a bounded region:
\label{thm:1}
\begin{enumerate}[(a)]
\item There are $356\,461$ Newton polytopes, up to symmetry, arising from tropical quartic surfaces.
\item Among these Newton polytopes, precisely $15\,139$ arise from smooth tropical quartic surfaces.
\item  The f-vectors of K3 polytopes are the triples $\left(v,\frac{3}{2}v,\frac{1}{2}v + 2\right)$ 
where $v \in \{4,6,\ldots,50\} \cup \{54,56,64\}$.
\item There are $36\,297\,333$ K3 polytopes with $v \leq 30 $ vertices.
They are dual to the regular unimodular central triangulations of the Newton polytopes in (b) 
that have normalized volume at most $30$.
\end{enumerate}
\end{theorem}

Our most relevant objects will be defined in Section~\ref{sec:zwei}.
Our computational proof of Theorem~\ref{thm:1}, presented in Section \ref{sec:hunt_K3},
 proceeds as follows. First of all, we list all lattice subpolytopes of $4\Delta_3$ that have
  an interior lattice point (Proposition~\ref{prop:canonical}). These are the Newton polytopes of tropical quartic surfaces
 that have a bounded region. If the quartic is also smooth, then that
    Newton polytope is a reflexive polytope (Proposition \ref{prop:bij_triangs}).
    The census of these polytopes is given in  Corollary \ref{cor:reflexive}.
By looking at the triangulations of these reflexive polytopes, we generate K3 polytopes. Specifically, the combinatorics of a K3 polytope is uniquely determined by the central part of a regular unimodular triangulation of $4\Delta_3$. We
 implemented a script to list  such triangulations in \polymake and \topcom.
 Table \ref{tab:dist_refl} summarizes the classification results we obtained.
 Full details and the source code for our computations are available at \url{https://github.com/gabrieleballetti/k3_polytopes}.

The bounded region of a tropical plane cubic identifies the j-invariant and hence
represents the curve in its tropical moduli space. Our ultimate hope for
K3 polytopes is that these can play a similar role for tropical moduli of quartic surfaces.
Our second result is a first step towards that goal. We study quartic surfaces whose Newton
polytope is one of the reflexive polytopes on our list.

The classical path towards moduli spaces is Geometric 
Invariant Theory \cite{Mum94}. In this setting one asks, for a given surface,
whether it is stable, semistable or unstable. We prove that all our quartic surfaces are stable,
provided their coefficients are generic relative to the reflexive Newton polytope.

\begin{theorem}
\label{thm:2}
Let $f \in \CC[x,y,z,w]$ be a homogeneous quartic whose Newton polytope
arises from a smooth tropical surface, as in Theorem~\ref{thm:1} (b).
Then the quartic surface $V(f)$ in $ \PP^3$ is stable.
\end{theorem}

The proof rests on Shah's characterization \cite{Shah} of stable quartic surfaces in terms of their singularities.
Our analysis of the singularities  exploits  results of Arnol\cprime d \cite{Arn72}
and  Mumford \cite{Mum77}. The latter allows us to check the stability only on reflexive polytopes that are minimal up to inclusion. 

\subsection*{Acknowledgements} We are very grateful to Michael Joswig for several inspiring discussions. We also thank Matteo Gallet, Lars Kastner and Benjamin Schr\"{o}ter for help with this project.
GB was partially supported by the Vetenskapsr\aa{}det grant NT:2014-3991. MP and BS acknowledge support by the Einstein Foundation Berlin, which also funded a visit of GB to TU Berlin.

\section{An Invitation to K3 polytopes}
\label{sec:zwei}

We begin with some basics from tropical geometry \cite{MS15}.
In the tropical semiring $\big(\mathbb{R} \cup \{\infty\}, \oplus, \odot\big)$, arithmetic is
  defined by $a \oplus b = \min \{a,b\}$ and $a \odot b = a+b$.  Consider  a tropical polynomial
\[f(x_1,  \ldots, x_n) \quad = \!\! \bigoplus_{\mathbf{v}=(v_1, \dots, v_n) \in \ZZ^n} \!\!\!\!\! c_{\mathbf{v}} \odot
 x_1^{\odot v_1} \odot \dots \odot x_n^{\odot v_n} \,\,\,= \,\,\,\bigoplus_{\mathbf{v} \in \ZZ^n} c_{\mathbf{v}} x^{\mathbf{v}}. \]
The \emph{tropical hypersurface $T(f)$} is defined as the set of points in $\RR^n$ at which the minimum among the quantities $c_{\mathbf{v}} + \mathbf{v} \cdot x$ is attained at least twice. The \emph{Newton polytope of $f$} is the lattice polytope 
\[ \text{Newt}(f) =  \text{conv} \big\{\mathbf{v} \, : \, c_{\mathbf{v}} \not = \infty \big\}. \]
Let $A = \text{Newt}(f) \cap \ZZ^n$ be its set of  lattice points. The coefficients of $f$ induce a \emph{regular subdivision} $\cT$ of $A$ by taking the convex hull in $\RR^{n+1}$ of the points $(\mathbf{v}, c_{\mathbf{v}})$ and projecting the lower faces to $\text{Newt}(f)$.
The coefficient vectors inducing the same subdivision $\cT$ 
form a relatively open polyhedral cone in $\RR^{|A|}$, called \emph{the secondary cone}.
 The tropical hypersurface $T(f)$ is dual to the subdivision $\cT$,
 and they determine each other \cite[Proposition 3.1.6]{MS15}. 
  We say that $T(f)$ is \emph{smooth} if the subdivision is a unimodular triangulation, i.e., 
  all simplices have normalized volume one.

The closures in $\RR^n$ of the connected components in
 the complement of a tropical hypersurface are called the \emph{regions of $T(f)$}. 
 These regions are convex polyhedra, either bounded or unbounded. 

Consider  the $(n+1)$-st dilation of the standard $n$-dimensional simplex, 
\[(n+1) \Delta_n \,\,=\,\,  \text{conv}\Big\{(0,0,\dots,0),(n{+}1,0,\ldots,0),(0,n{+}1, \ldots,0),(0,0,\ldots,n{+}1)\Big\}.\]
It has a unique interior lattice point $\mathbf{p} = (1,1,\ldots, 1)$.
  Let $T(f)$ be a smooth tropical hypersurface in $\RR^n$ of degree $n+1$. The Newton polytope $\text{Newt}(f)$ is contained in $(n+1) \Delta_n$. If the interior of $\text{Newt}(f)$ contains the point $\mathbf{p}$, then the 
  hypersurface $T(f)$ has   a bounded region in its complement.

The case $n=2$ corresponds to cubic curves \cite{KMM08, Vig09}.
We are here interested in the case $n=3$:
\[ f(x,y,z) \quad = \,\,\bigoplus_{i+j+k\leq 4} c_{ijk} \odot x^{\odot i} \odot y^{\odot j} \odot z^{\odot k}. \]
Suppose that the  tropical quartic surface $T(f)$ is smooth.
The Newton polytope $\text{Newt}(f)$ is a lattice polytope inside $4\Delta_3$.
We assume that it has $\mathbf{p}$ in its interior, so there is a bounded region.

\begin{definition} A $3$-dimensional polytope $\cP$ is a \emph{K3 polytope} if it is the closure of the unique bounded region in the complement of a smooth tropical quartic surface in $\RR^3$.
\end{definition} 

Every K3 polytope has a rational normal fan. This fan is simplicial because $T(f)$ is smooth.
Hence a K3 polytope is always simple, i.e.~each of its vertices is contained in exactly three edges. 

\begin{example}\label{es:bigK3} 
The following tropical polynomial defines a smooth tropical quartic surface:
$$ \begin{matrix}
f(x,y,z) & = & 5(x^4 \oplus y^4 \oplus z^4) \oplus 3(x^3y \oplus x^3z \oplus xy^3 \oplus y^3z \oplus xz^3 \oplus yz^3) \oplus 2( x^2y^2 \oplus x^2z^2 \oplus y^2z^2) \\ & &  \!\!
\oplus \,0( x^2yz \oplus xy^2z \oplus xyz^2 ) \oplus 3(x^3 \oplus y^3 \oplus  z^3) \oplus 0(x^2y \oplus x^2z \oplus xy^2 \oplus  y^2z \oplus xz^2 \oplus yz^2) \\ & & 
\oplus \,\,2(x^2 \oplus y^2 \oplus z^2) \,\oplus\,  0(xy \oplus xz   \oplus yz)\, \oplus \,3(x \oplus y \oplus z) \oplus  (-9xyz) \oplus 5.
\end{matrix}
$$
Note that ${\rm Newt}(f) = 4 \Delta_3$ has $\mathbf{p}$ in its interior.
The surface $T(f)$ is shown on the right in Figure \ref{fig:bigK3}, and its
 K3 polytope is shown on the left in Figure \ref{fig:bigK3}.
It is simple and has the f-vector $(64, \, 96, \, 34)$.  
\end{example}

\begin{center}
\begin{figure}[h]
 \includegraphics[width=5.6cm]{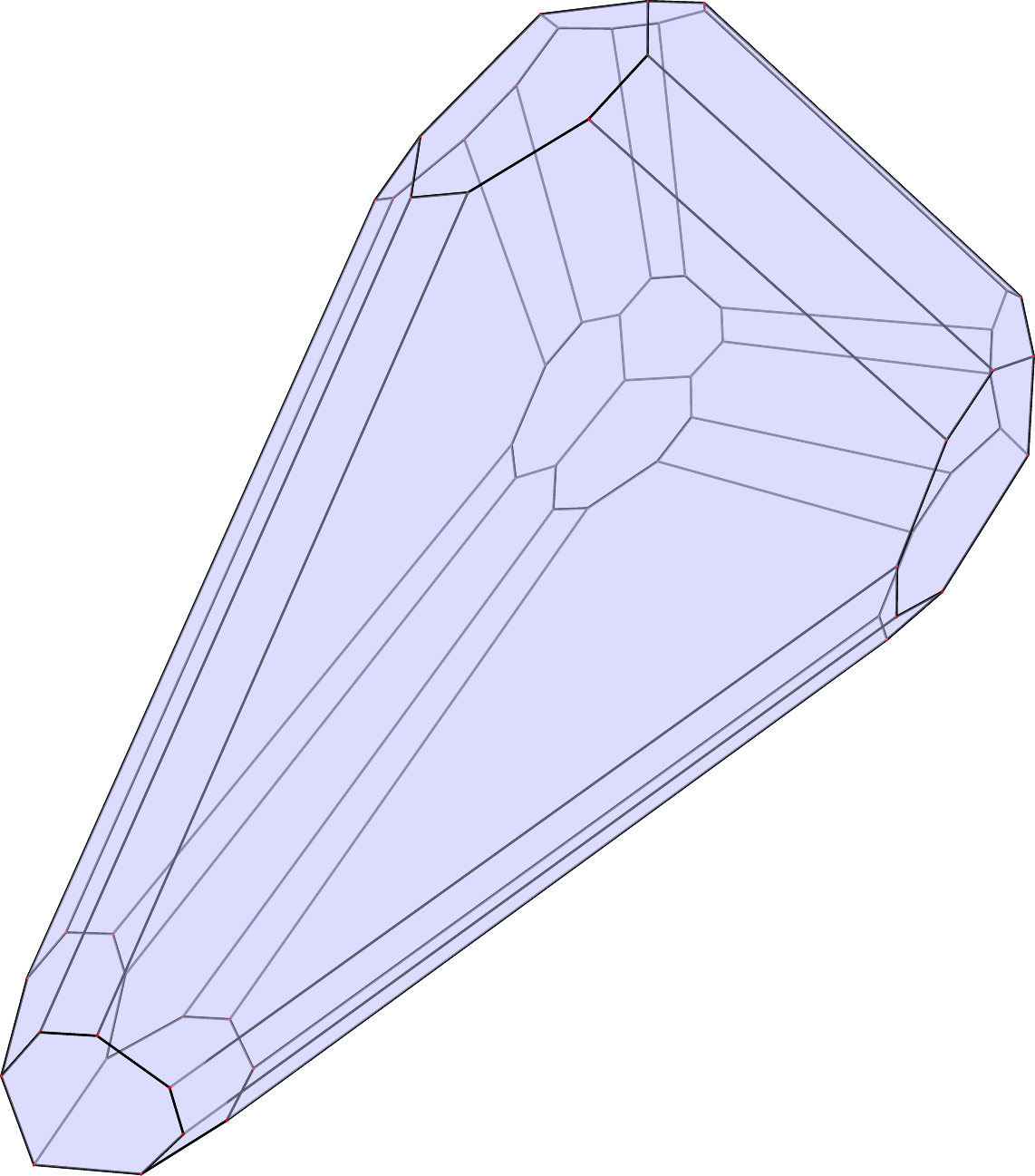}   \qquad
\includegraphics[width=6.5cm]{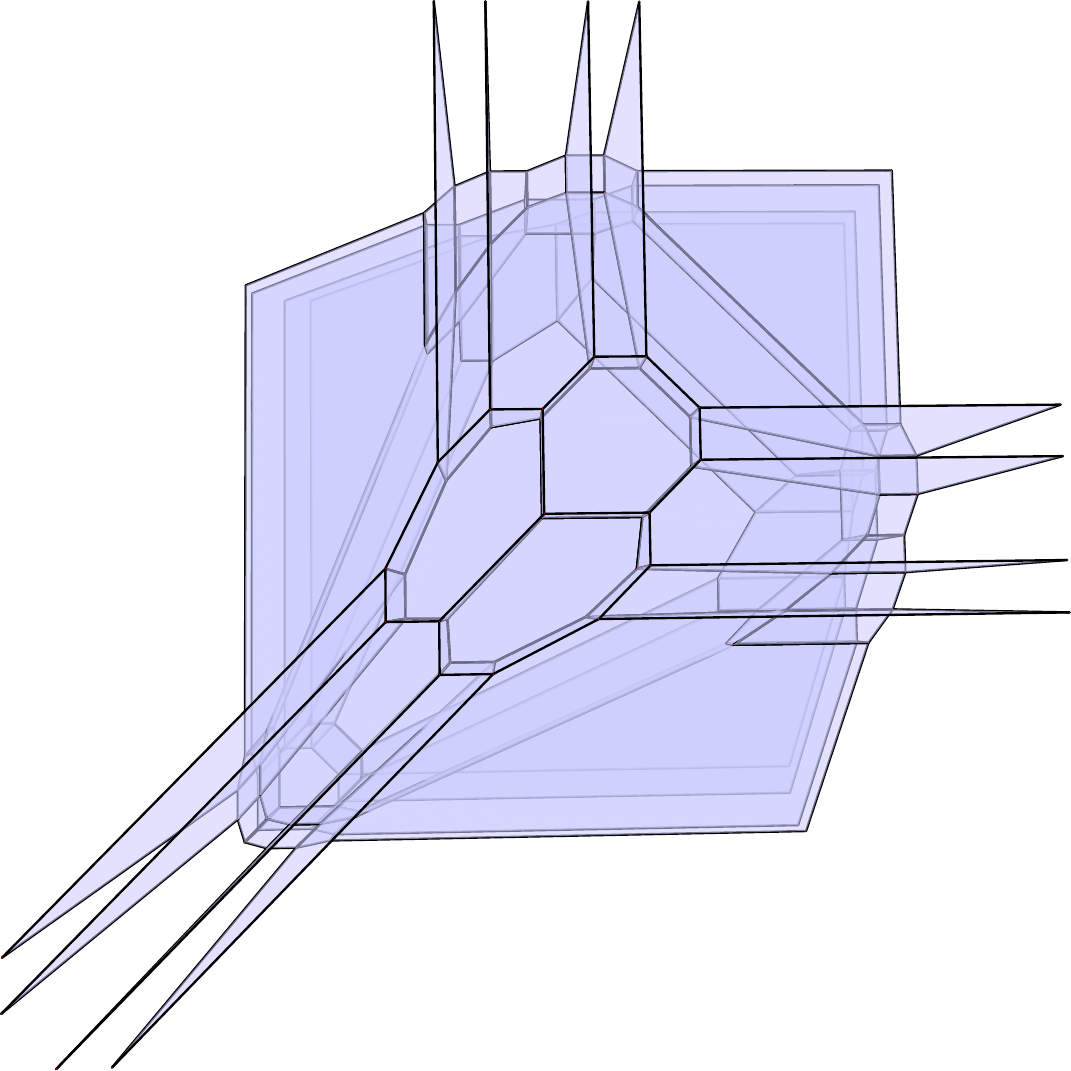} 
\caption{\label{fig:bigK3} The K3 polytope (left) arising from a tropical quartic surface (right).
} 
\end{figure}
\end{center}

\section{The hunt for K3 polytopes}\label{sec:hunt_K3}

We are interested in classifying K3 polytopes.   They are
 dual to regular unimodular triangulations of their Newton polytope. We first focus on the latter objects.
For basic definitions on triangulations we refer to \cite{DRS10}. By a {\em triangulation} of a lattice polytope $P$ we mean a triangulation of the point configuration given by the lattice points in $P$. Unimodular triangulations are particular \emph{fine triangulations}, i.e.~they do not admit any proper refinement. Borrowing some vocabulary from toric geometers, we say that a lattice polytope is \emph{canonical} if it has just one lattice point in its relative interior.
This corresponds to a toric Fano variety with at worst canonical singularities. Note that we do not assume that the interior point is the origin of the lattice, as it is usually assumed in the literature. 
If $P$ is a canonical polytope with interior lattice point $\mathbf{p}$ and $\cT$ is a triangulation of $P$, then the \emph{central part} of $\cT$ consists of the simplices of $\cT$ whose union with $\mathbf{p}$ is a simplex in $\cT$. This is 
also known as the \emph{star of $\mathbf{p}$ in $\cT$}. If $\cT$ coincides with its central part, then we call $\cT$ \emph{central}.

Any triangulation $\cT$ of $P$ induces a triangulation $\{ S \cap  \partial P \st S \in \cT \}$ of the boundary $\partial P$ of $P$. Conversely, any triangulation $\tau$ of $\partial P$ induces a central triangulation $\{ S \cup \{ \mathbf{p} \} \st S \in \tau \}$ of $P$. Thus, 
\begin{equation}
\label{eq:bij_triangs}
\{ \cT \st \text{$\cT$ is a central triangulation of $P$} \}\,\, \longleftrightarrow \,\,
\{ \tau \st \text{$\tau$ is a triangulation of $\partial P$} \}
\end{equation}
is a bijection.
The K3 polytope of $f$ is determined by the central part of a regular unimodular triangulation 
of ${\rm Newt}(f)$.
We thus ignore all triangulations that are not central. Indeed, the central part of such a triangulation 
will arise as a central triangulation of a smaller Newton polytope.

In the following subsections we  construct and classify K3 polytopes  as follows:
\begin{enumerate}
\item[\ref{ssec:canonical}] using \polymake \cite{polymake} we list all lattice polytopes $P \subseteq 4 \Delta_3$
with $\mathbf{p} \in {\rm int}(P)$ as above;
\item[\ref{ssec:reflexive}] we extract a sublist of those polytopes that admit a unimodular central triangulation;
\item[\ref{ssec:regular}] using \topcom \cite{topcom} we explore the regular unimodular central triangulations
 of the polytopes in the sublist above; each such triangulation determines one K3 polytope;
\item[\ref{ssec:fvectors}] the possible f-vectors of a K3 polytopes are described.
\end{enumerate}

\subsection{Newton polytopes of tropical quartic surfaces}
\label{ssec:canonical}
One can find the set $\mathcal{S}$ of all canonical Newton polytopes of 
quartic surfaces by starting from $4\Delta_3$ and progressively removing a vertex.

\begin{algorithm}
\label{al:trim}
\hspace{1pt}\\
INPUT: The polytope $4\Delta_3$.\\
OUTPUT: The set $\mathcal{S}$ of all $3$-dimensional canonical subpolytopes of $4\Delta_3$.
\begin{enumerate}[1.]
\item Set $\mathcal{S} \coloneqq \{ 4\Delta_3 \}$.
\item For $P \in \mathcal{S}$ and  each vertex $v$ of $P$, let $P_v \coloneqq \conv (P \cap \ZZ^3 
\backslash \{ v \})$. If ${\bf p} \in {\rm int}(P_v)$, add $P_v$ to $\mathcal{S}$.
\item If at least one polytope has been added to $\mathcal{S}$ in the last step then repeat step 2.
\end{enumerate}
\end{algorithm}

Our implementation of Algorithm~\ref{al:trim} in \polymake leads to the following result.

\begin{proposition}\label{prop:canonical}
Up to symmetry there are $356\,461$ canonical Newton polytopes of quartic surfaces.
\end{proposition}

This proves Theorem~\ref{thm:1} (a).
Kasprzyk  \cite{Kas10}  classified all  $3$-dimensional canonical polytopes, so one could have attempted to deduce 
Proposition~\ref{prop:canonical} from his list. Kasprzyk's classification is 
up to affine unimodular equivalence, while for us it is preferable to work modulo
the symmetric group $S_4$. For this reason it is easier to generate
 all canonical subpolytopes of $4\Delta_3$ from scratch, via Algorithm~\ref{al:trim}.

\subsection{Reflexive Newton polytopes}
\label{ssec:reflexive}

We next incorporate the requirement that the tropical quartic surface is smooth.
Let $P \subset \RR^d$ be a $d$-dimensional lattice polytope with $k$ facets. We can write
\[
P = \{ \mathbf{x} \in \RR^d \st 
 A \mathbf{x} \geq \mathbf{c}
\},
\]
where $A$ is $k \times d$-matrix whose rows are primitive vectors in $\ZZ^d$
and ${\bf c} \in \ZZ^k$.
Suppose that $P$ has one interior lattice point $\mathbf{p}$. We say that $P$ is \emph{reflexive} if $A \mathbf{p} - \mathbf{c} = \mathbf{1}$, where $\mathbf{1}$ is the all-one vector $(1,\ldots,1)^t$. Reflexive polytopes are those canonical polytopes
where {\bf p} is in an adjacent lattice hyperplane to any facet.
They were introduced by Batyrev \cite{Bat94} within mirror symmetry and by Hibi \cite{Hib92} within combinatorics.
The polar of a reflexive polytope is again a lattice polytope, and it corresponds to
a Gorenstein toric Fano variety. If $P$ is reflexive, then the bijection \eqref{eq:bij_triangs} 
restricts to a bijection on  unimodular triangulations. Every fine triangulation of $\partial P$
is unimodular, since this holds for lattice polygons.
By putting these facts together, we obtain the following characterization.

\begin{proposition} \label{prop:bij_triangs}
A $3$-dimensional canonical lattice polytope $P$ is reflexive if and only if every central fine triangulation of $P$ is unimodular. 
\end{proposition}

\begin{table}[h]
\centering
\includegraphics{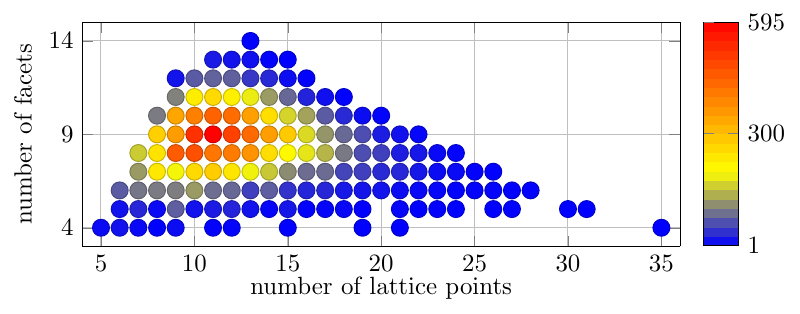}
\caption{\label{tab:dist_refl} Reflexive $3$-polytopes in $4 \Delta_3$  by number of facets and lattice points.}
\end{table}
 
We use Proposition~\ref{prop:canonical} to extract the list of $3$-dimensional reflexive subpolytopes of $4\Delta_3$. 
Note that reflexive polytopes up to dimension $4$ are fully classified in \cite{KS98,KS00}, but, as in the previous 
subsection, we work only up to $S_4$-symmetry, and it is easier to obtain complete lists from Proposition~\ref{prop:canonical}.

\begin{corollary}
\label{cor:reflexive}
Up to $S_4$-symmetry there are $15\,139$ reflexive  $3$-polytopes which are contained in $4\Delta_3$.
\end{corollary}

The distribution of our reflexive $3$-polytopes by their number of lattice points and
their number of facets is  shown in a heat map in Table~\ref{tab:dist_refl}.
This computation proves Theorem~\ref{thm:1} (b).

\subsection{Regular triangulations}
\label{ssec:regular}

We now apply \topcom \cite{topcom} to the list of polytopes  in Corollary~\ref{cor:reflexive}.
Let $P$ be one of them. We first compute all unimodular central triangulations of  $P$,
and then we filter out the non-regular ones using \topcom. One can find all unimodular central triangulations of $P$ simply by iterating over all unimodular triangulations of each facet of $P$. The union of such triangulations is a unimodular triangulation of $\partial P$ which induces a unimodular central triangulation of $P$. 
For each such triangulation of $P$, we then check for regularity with \polymake.

The number of regular triangulations appears to grow exponentially with the number of lattice points 
(see Table~\ref{tab:regtriangs_refl}), making this classification infeasible. We computed all regular 
unimodular central triangulations
for more than $96 \%$ of the total number of reflexive polytopes of Corollary~\ref{cor:reflexive}. 
We stopped after this, as a complete classification is out of reach. In total we calculated 
 $36 \; 297 \; 333$ different regular unimodular central triangulations, each of them 
 corresponding to a K3 polytope.
\begin{table}[h]
\centering
\includegraphics{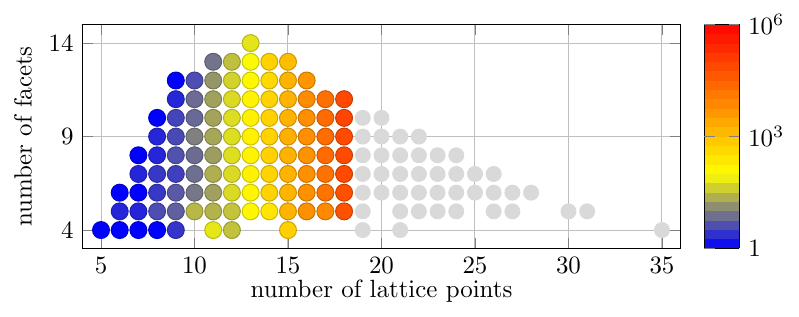}
\vspace{-.2cm}
\caption{\label{tab:regtriangs_refl}
The average number of regular unimodular central triangulations 
of the reflexive polytopes represented by each of the points of Table~\ref{tab:dist_refl}.}
\end{table}

Table \ref{tab:regtriangs_refl} indicates the number of regular unimodular central 
triangulations of our reflexive polytopes with up to $18$ lattice points. 
This constraint is equivalent to having normalized volume at most $30$. 
From this we obtain part (d) of Theorem~\ref{thm:1}, here restated as follows:

\begin{corollary}\label{cor:partial}
The reflexive polytopes of volume $\leq 30$ in Corollary~\ref{cor:reflexive} admit 
a total of $36\,297\,333$ regular unimodular central triangulations.
 Every K3 polytope with $\leq 30$ vertices  arises from one of these. 
In the table below, these K3 polytopes are counted 
according to their numbers of vertices:
\end{corollary}
\begin{center}
\begin{tabular}{ |c|c||c|c| } 
\hline
\# vertices & \# triangulations & \# vertices & \# triangulations \\
\hline
4	&	$9$			&	18	&	$106 \; 049$		\\
6	&	$117$		&	20	&	$266 \; 206$		\\
8	&	$561$		&	22	&	$634 \; 228$		\\
10	&	$2 \; 065$	&	24	&	$1 \; 582 \; 156$	\\
12	&	$6 \; 261$ 	&	26	&	$3 \; 564 \; 613$	\\
14	&	$16 \; 523$	&	28	&	$9 \; 341 \; 111$	\\
16	&	$42 \; 780$	&	30	&	$20 \; 734 \; 654$	\\
\hline
\end{tabular}
\end{center}

\bigskip

We also examined the different combinatorial types of K3 polytopes with up to $18$ vertices:

\begin{corollary} \label{cor:comb_types} The K3 polytopes with most $18$ vertices have 
$764$ distinct vertex-facet incidence graphs. Their number, for each of the possible numbers of vertices, is listed in the table below:
\end{corollary}
\begin{center}
\begin{tabular}{ |c|c||c|c| } 
\hline
\# vertices & \# incidence garphs & \# vertices & \# incidence graphs \\
\hline
4	&	$1$		&	12	&	14		\\
6	&	$1$		&	14	&	44	\\
8	&	$2$		&	16	&	158	\\
10	&	$5$	&	18	&	539	\\
\hline
\end{tabular}
\end{center}
%
%

\subsection{f-vectors of K3 polytopes}
\label{ssec:fvectors}

The f-vector of a $3$-dimensional polytope  is the triple $f= (f_0, f_1, f_2)$
where $f_0, f_1$ and $f_2$ are the numbers of its  vertices, edges and facets. 
Euler's relation states that $f_0-f_1+f_2 = 2$. If the polytope is simple
then $3 f_0 = 2 f_1$. This holds for K3 polytopes.

The f-vector of a K3 polytope depends only on the polytope $P$ inside $4\Delta_3$ from which it originates.
Namely, $f_i$ counts the $(3-i)$-dimensional interior simplices in a unimodular triangulation of~$P$.

\begin{lemma}
\label{lem:fvector}
Consider the K3 polytope dual to a  regular unimodular central triangulation
of a relexive polyope $P$ in $4 \Delta_3$. The entries of the f-vector 
of this K3 polytope are
\[f_0 = \Vol(P) ,\qquad f_1 = \frac{3}{2} \Vol(P), \qquad f_2 = \frac{1}{2} \Vol(P) + 2.
\]
In particular, every K3 polytope has an even number of vertices.
\end{lemma}
 
Theorem~\ref{thm:1} (c) now follows from
Lemma~\ref{lem:fvector} together with the census in Corollary~\ref{cor:reflexive}.
 Table \ref{tab:fvectors} comprises all the possible f-vectors that a K3 polytope can have.
 Each f-vector appears together with the number $\# P$ of relexive polytopes 
 it arises from. These numbers add up to  15 139.

\begin{table}[h]
\begin{tabular}{ |c|c||c|c||c|c| }
 \hline
f-vector & \# $P$ & f-vector & \# $P$ & f-vector & \# $P$ \\
\hline
(4, 6, 4) & 9 & (22, 33, 13) & 1248  &  (40, 60, 22) & 27 \\

(6, 9, 5) & 102 & (24, 36, 14) & 922  & (42, 63, 23) & 18 \\

(8, 12, 6) & 412 & (26, 39, 15) & 628  & (44, 66, 24) & 7 \\

(10, 15, 7) & 959 & (28, 42, 16) & 465  & (46, 69, 25) & 9 \\

(12, 18, 8) & 1642 & (30, 45, 17) & 295 & (48, 72, 26) & 2 \\ 
(14, 21, 9) & 2083 & (32, 48, 18) & 203 & (50, 75, 27) & 2 \\

(16, 24, 10) & 2194 &  (34, 51, 19) & 128  & (54, 81, 29) & 1 \\ 
(18, 27, 11) & 1997 & (36, 54, 20) & 85  & (56, 84, 30) & 1 \\
(20, 30, 12) & 1646 & (38, 57, 21) & 53 & (64, 96, 34) & 1 \\
 \hline
\end{tabular}
\smallskip
\caption{\label{tab:fvectors}Admissible f-vectors of K3 polytopes.}
\end{table}
 
\section{Singularities of quartic surfaces}
\label{sec:four}

We now leave the tropical setting, and we consider 
(the moduli space of) quartic surfaces in complex projective space.
We shall examine our reflexive polytopes through the lens of Geometric Invariant Theory \cite{Mum94}. We study general quartic surfaces with a fixed reflexive Newton polytope.

Consider the space $ \text{HS}_{4,3} = \CC[x,y,z,w]_4$ of all quartic polynomials
with complex coefficients,
\[f(x, y, z, w) \quad = \sum_{i+ j+ k\leq 4} c_{ijk} x^i y^j z^k w^{4-i-j-k}.\]
The variety $V(f)$ defined by such a polynomial is a quartic surface in $\PP^3$.
We write $\mathbb{HS}_{4,3}$ for the $34$-dimensional projective space $\PP(\text{HS}_{4,3})$
of all quartic surfaces.
 The special linear group $\text{SL}(4)$ acts on $\mathbb{HS}_{4,3}$,
 and on the associated polynomial ring $\CC[\HS_{4,3}]$, generated by $35$ unknowns  $C  = (c_{ijk})$.

\begin{definition} Let $F(C)$ be a  polynomial in the $\CC$-algebra $\CC[\HS_{4,3}]$. Then $F(C)$ is called \emph{invariant} if $g.F(C)  = F(C)$ for all $g \in \text{SL}(4)$. We denote by $\CC[\HS_{4,3}]^{\text{SL(4)}}$ the subalgebra of invariants. 
\end{definition}

The \emph{moduli space of quartic surfaces in $\PP^3$} is the projective variety 
determined by this invariant ring, namely $\,{\rm Proj} \bigl(\CC[\HS_{4,3}]^{\text{SL(4)}}\bigr)$.
Following Mumford \cite{Mum77, Mum94}, we give the following definitions. 

\begin{definition} Let $f$ be an element of $\HS_{4,3}$. We say that 
\begin{itemize}
\item $f$ is \emph{stable} if the orbit $O(f)^{\text{SL}(4)}$ is closed and the stabilizer $\text{stab}(f)$ is finite;
\item $f$ is \emph{semistable} if the closure of the orbit $O(f)^{\text{SL}(4)}$ does not contain the point $0$;
\item $f$ is \emph{unstable} if the closure of the orbit $O(f)^{\text{SL}(4)}$ contains the point $0$.
\end{itemize}
\end{definition}
We use the notation $\HS_{4,3}^{s}$ and $\HS_{4,3}^{ss}$ to denote the set of stable and semistable points
respectively.
The {\em GIT quotient} of the action of $\text{SL}(4)$ is defined on the semistable locus $\HS_{4,3}^{ss}$, as follows:
\[\varphi \, : \, \HS_{4,3}^{ss} \,\,\rightarrow \,\,\HS_{4,3}/\!/ \text{SL}(4)\,:= \,\text{Proj}\bigl(\CC[\HS_{4,3}]^{\text{SL(4)}}\bigr).  \]
 The image $\varphi(\HS_{4,3}^s)$ of the stable locus is the \emph{moduli space of stable quartic surfaces in $\PP^3$}.

Determining stable and semistable points is therefore a key step in the construction of moduli spaces via
   Geometric Invariant Theory. This task is deeply connected with the study of singularities. For example, it is 
   known that all nonsingular hypersurfaces in $\PP^n$ of
    degree $d \geq 2$ (respectively $d\geq 3$) are semistable (stable). This follows from the fact that the discriminant 
     is invariant under the action of $\text{SL}(n+1)$. Another classical result states that a plane cubic $(d=3,n=2)$
       is unstable if the curve has a triple point, a cusp or two components tangent to a point.
If it has an ordinary double point then it is semistable but not stable.
The right diagram in Figure \ref{fig:mumford} indicates this.

In \cite{Shah}, Shah describes whether a quartic surface is stable, semistable or unstable by looking at the type of singularities it has. As it is summarized in \cite{Mum77}, a quartic surface is stable if and only if
\begin{itemize}
\item its singular locus contains at most rational double points, and ordinary double curves possibly with pinch points, but not double lines;
\item  in the case when the singular locus is reducible, there is no plane as component,
 and there are no multiple components.
\end{itemize} 


The Hilbert-Mumford Criterion \cite[Theorem 2.1]{Mum94} states that, after a linear change of coordinates,
 the stability of a surface $f$ in $\HS_{4,3}$ can be checked by looking at its Newton polytope Newt$(f)$.
Following  \cite[\S 1.14]{Mum77}, we must look at the planes that pass through the point $\mathbf{p} = (1,1,1,1)$.

\begin{theorem}[Mumford \cite{Mum77}]
	
\label{thm:mumford}
A point $f$ in $\HS_{4,3}$ is stable if and only if, 
for every choice of coordinates, and for all planes $H$ through $\mathbf{p}$,  
each open halfspace of $H$ contains a monomial of $f$.
\end{theorem}

In other words, $f$ is stable if, for every choice of coordinates and all planes $H$, the Newton polytope Newt$(f)$ does not entirely lie in one of the two closed halfspaces defined by $H$. This situation is depicted
for plane cubics  in Figure~\ref{fig:mumford}. As a consequence we have the following corollary.

\vspace{-0.2in}
\begin{figure}[h]
\includegraphics[width=6cm]{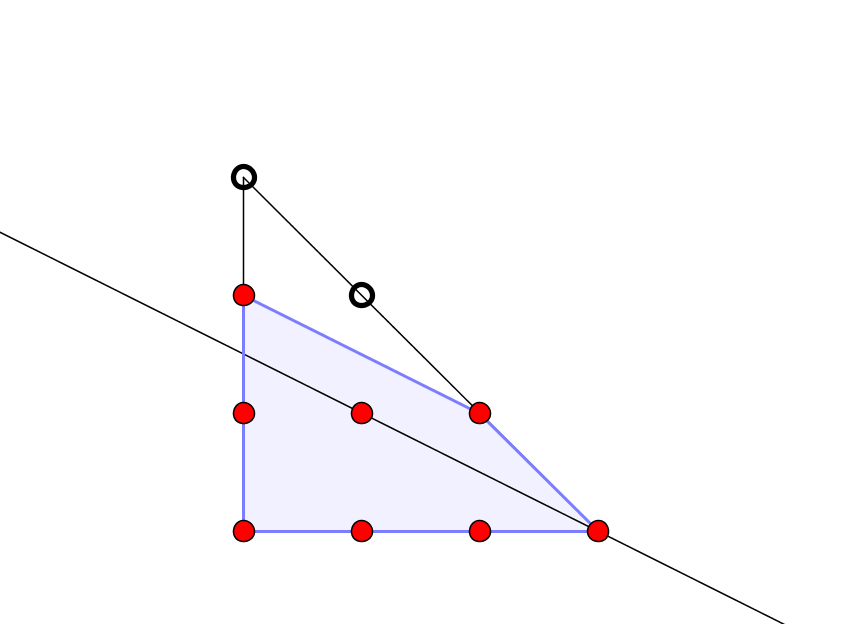} \qquad \qquad  \includegraphics[width=6cm]{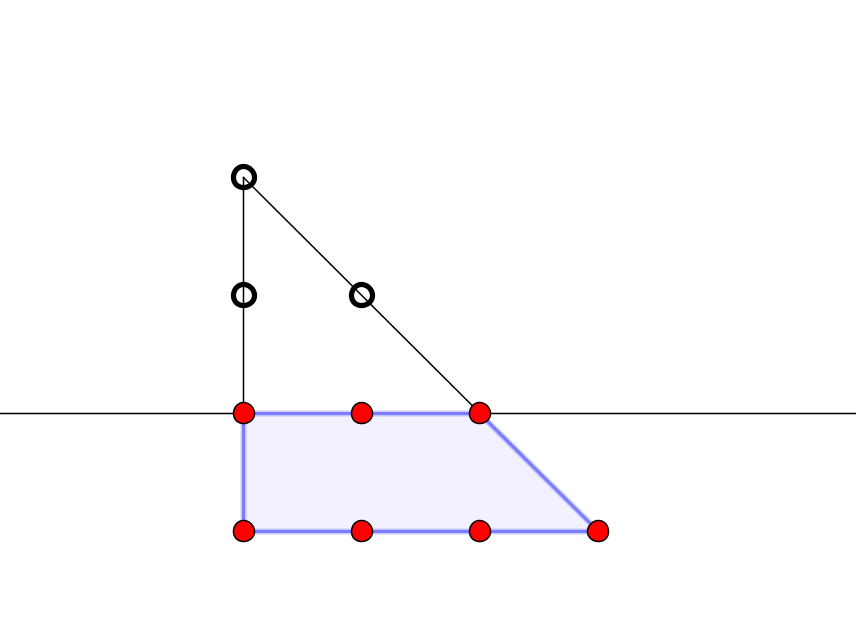} 
\caption{\label{fig:mumford} The Newton polytopes of a stable and of a semistable cubic plane curve.} 
\end{figure}

\begin{corollary}
\label{cor:stability}
Let $f,g \in \HS_{4,3}$ such that $\Newt(f) \subseteq \Newt(g)$. 
If $f$ is stable and $g$ has general coefficients then $g$ is stable.
\end{corollary}

This allows us to restrict our interest to polytopes that satisfy the following minimality condition. 
A reflexive lattice polytope $P$ contained in $4\Delta_3$ is 
called \emph{minimal} if it does not properly contain any reflexive polytopes.
We note that this notion  cannot be extended naturally to canonical polytopes, as there are reflexive polytopes which are minimal, but properly contain canonical polytopes. In this sense, the above definition differs from the
notion of minimality used by Kasprzyk in \cite{Kas10}.

\section{Minimal polytopes}

In this section we classify minimal polytopes and we examine their combinatorial properties.
The stability of their quartic surfaces will be established in the next section. From the list of reflexive polytopes in
Corollary~\ref{cor:reflexive}, we can extract all those that are minimal up to inclusion.

\begin{proposition}
\label{prop:minimal}
Up to the $S_4$-action, there are precisely $115$ minimal reflexive polytopes
in $4 \Delta_3$. Among these, $29$ admit two regular unimodular central triangulations,
and $86$ admit just one.  
\end{proposition}

We next describe the combinatorics of all K3 polytopes that arise
from the triangulations in Proposition~\ref{prop:minimal}.
Each minimal polytope contributes one or two K3 polytopes to the following census.
We obtain six combinatorial types of  K3 polytopes, each displayed
by a vertex-facet incidence list.

\begin{itemize}
\item {\em K3 polytopes from minimal polytopes with $5$ lattice points}:
Each of the four tetrahedra in the triangulation shares a facet with the others. The K3 polytope is a tetrahedron:
\[ \{0 \, 1 \,  2\} \ \{1 \, 2 \, 3\} \ \{0 \, 1 \, 3\} \ \{0 \, 2 \, 3\}.\] 
\item  {\em K3 polytopes from minimal polytopes with $6$ lattice points}: 
The minimal polytope is a bipyramid.
The K3 polytope is a triangular prism. It  has the f-vector $(6,9,5)$:
\[ \{2 \, 3 \,  4 \, 5\} \ \{0 \, 1 \, 4 \, 5\} \ \{0 \, 1 \, 2 \, 3\} \ \{1 \, 3 \, 5\} \ \{0 \, 2 \, 4\}\] 
\item 
{\em K3 polytopes from minimal polytopes with $7$ lattice points}: The minimal 
polytope is an octahedron. The K3 polytope has the f-vector $(8,12,6)$. Combinatorially, it is a cube:
\[ \{4 \, 5 \, 6 \,  7 \} \ \{2 \, 3 \, 6 \, 7 \} \ \{1 \, 3 \, 5 \, 7 \} \ \{0 \, 2 \, 4 \, 6 \} \ \{0 \, 1 \, 4 \, 5 \} \ \{0 \, 1 \, 2 \, 3 \} \]
\item {\em K3 polytopes from minimal polytopes with $8$ lattice points}:
These K3 polytopes are pentagonal prisms, so they have the f-vector $(10, 15,7)$:
\[ \{0 \, 1 \, 2 \, 3 \, 4 \} \ \{0 \, 4 \, 5 \, 9\} \ \{3 \, 4 \, 8 \, 9\} \ \{2 \, 3 \, 7 \, 8\} \ \{5 \, 6 \, 7 \, 8 \, 9\} \ \{0 \, 1 \, 5 \, 6\} \ \{1 \, 2 \, 6 \, 7\} \]
\item {\em K3 polytopes from minimal polytopes with $9$ lattice points}:
These have f-vector $(12, 18, 8)$, and they come in two combinatorial types.  
The first is a hexagonal prism:
\[\{1 \, 2\, 4\, 7\, 8\, 10\} \ \{0\, 1\, 6\, 7\} \ \{0\, 1\, 2\, 3\} \ \{4\, 5\, 10\, 11\} \ \{0\, 3 \, 5 \, 6 \, 9 \,11\} \ \{8\, 9\, 10\, 11\} \ \{2\, 3\, 4\, 5\} \ \{6\, 7\, 8\, 9\}. \]
The second (lower right in Figure~\ref{fig:K3minimal})
has two triangles and two heptagons among its facets:
\[\{4 \, 5\, 6\, 7\, 8\, 9\, 11\} \ \{2\, 3\, 4\, 8\} \ \{0\, 1\, 3\, 4\} \ \{8\, 9\, 10\, 11\} \ \{1\, 2 \, 3 \, 5 \, 7 \, 8 \,11\} \ \{2\, 10\, 11\} \ \{0\, 1\, 5\} \ \{6\, 7\, 8\, 9\}. \]
\end{itemize}
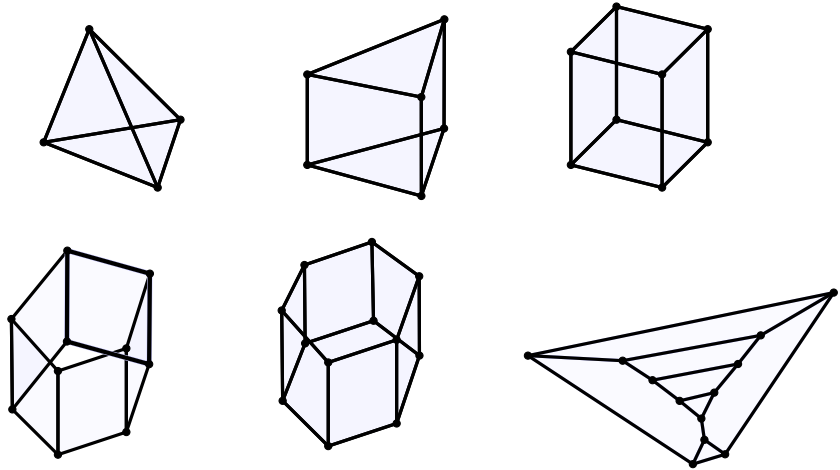
\begin{figure}[H]
\begin{subfigure}[b]{0.2\textwidth}
\tikzset{every picture/.style={scale=0.3}}%
\definecolor{ccccff}{rgb}{0.8,0.8,1.}
\begin{tikzpicture}[line cap=round,line join=round,>=triangle 45,x=1.0cm,y=1.0cm]
\clip(-2.161700015990635,-0.8199126924131027) rectangle (19.019202164895066,9.266830855083576);
\fill[line width=2.pt,color=ccccff,fill=ccccff,fill opacity=0.10000000149011612] (0.,3.) -- (2.,8.) -- (6.,4.) -- cycle;
\fill[line width=2.pt,color=ccccff,fill=ccccff,fill opacity=0.10000000149011612] (2.,8.) -- (6.,4.) -- (5.,1.) -- cycle;
\fill[line width=2.pt,color=ccccff,fill=ccccff,fill opacity=0.10000000149011612] (2.,8.) -- (0.,3.) -- (5.,1.) -- cycle;
\fill[line width=2.pt,color=ccccff,fill=ccccff,fill opacity=0.10000000149011612] (0.,3.) -- (6.,4.) -- (5.,1.) -- cycle;
\draw [line width=1.2pt] (0.,3.)-- (2.,8.);
\draw [line width=1.2pt] (2.,8.)-- (6.,4.);
\draw [line width=1.2pt] (6.,4.)-- (0.,3.);
\draw [line width=1.2pt] (2.,8.)-- (6.,4.);
\draw [line width=1.2pt] (6.,4.)-- (5.,1.);
\draw [line width=1.2pt] (5.,1.)-- (2.,8.);
\draw [line width=1.2pt] (2.,8.)-- (0.,3.);
\draw [line width=1.2pt] (0.,3.)-- (5.,1.);
\draw [line width=1.2pt] (5.,1.)-- (2.,8.);
\draw [line width=1.2pt] (0.,3.)-- (6.,4.);
\draw [line width=1.2pt] (6.,4.)-- (5.,1.);
\draw [line width=1.2pt] (5.,1.)-- (0.,3.);
\begin{scriptsize}
\draw [fill=black] (0.,3.) circle (4.5pt);
\draw [fill=black] (2.,8.) circle (4.5pt);
\draw [fill=black] (6.,4.) circle (4.5pt);
\draw [fill=black] (5.,1.) circle (4.5pt);
\end{scriptsize}
\end{tikzpicture}
\end{subfigure} \quad
\begin{subfigure}[b]{0.2\textwidth}
\tikzset{every picture/.style={scale=0.3}}%
\definecolor{ccccff}{rgb}{0.8,0.8,1.}
\begin{tikzpicture}[line cap=round,line join=round,>=triangle 45,x=1.0cm,y=1.0cm]
\clip(-2.1617000159906348,-0.8199126924131042) rectangle (19.019202164895066,9.266830855083574);
\fill[line width=2.pt,color=ccccff,fill=ccccff,fill opacity=0.10000000149011612] (0.,2.0008495480853425) -- (5.998362179737032,3.6127136855130266) -- (5.003539782418382,0.6282464935570807) -- cycle;
\fill[line width=2.pt,color=ccccff,fill=ccccff,fill opacity=0.10000000149011612] (0.,6.) -- (6.0109548683106855,8.435713409222425) -- (5.,5.) -- cycle;
\fill[line width=2.pt,color=ccccff,fill=ccccff,fill opacity=0.10000000149011612] (0.,6.) -- (5.,5.) -- (5.003539782418382,0.6282464935570807) -- (0.,2.0008495480853425) -- cycle;
\fill[line width=2.pt,color=ccccff,fill=ccccff,fill opacity=0.10000000149011612] (6.0109548683106855,8.435713409222425) -- (5.998362179737032,3.6127136855130266) -- (5.003539782418382,0.6282464935570807) -- (5.,5.) -- cycle;
\fill[line width=2.pt,color=ccccff,fill=ccccff,fill opacity=0.10000000149011612] (6.0109548683106855,8.435713409222425) -- (5.998362179737032,3.6127136855130266) -- (0.,2.0008495480853425) -- (0.,6.) -- cycle;
\draw [line width=1.2pt] (0.,2.0008495480853425)-- (5.998362179737032,3.6127136855130266);
\draw [line width=1.2pt] (5.998362179737032,3.6127136855130266)-- (5.003539782418382,0.6282464935570807);
\draw [line width=1.2pt] (5.003539782418382,0.6282464935570807)-- (0.,2.0008495480853425);
\draw [line width=1.2pt] (0.,6.)-- (6.0109548683106855,8.435713409222425);
\draw [line width=1.2pt] (6.0109548683106855,8.435713409222425)-- (5.,5.);
\draw [line width=1.2pt] (5.,5.)-- (0.,6.);
\draw [line width=1.2pt] (0.,6.)-- (5.,5.);
\draw [line width=1.2pt] (5.,5.)-- (5.003539782418382,0.6282464935570807);
\draw [line width=1.2pt] (5.003539782418382,0.6282464935570807)-- (0.,2.0008495480853425);
\draw [line width=1.2pt] (0.,2.0008495480853425)-- (0.,6.);
\draw [line width=1.2pt] (6.0109548683106855,8.435713409222425)-- (5.998362179737032,3.6127136855130266);
\draw [line width=1.2pt] (5.998362179737032,3.6127136855130266)-- (5.003539782418382,0.6282464935570807);
\draw [line width=1.2pt] (5.003539782418382,0.6282464935570807)-- (5.,5.);
\draw [line width=1.2pt] (5.,5.)-- (6.0109548683106855,8.435713409222425);
\draw [line width=1.2pt] (6.0109548683106855,8.435713409222425)-- (5.998362179737032,3.6127136855130266);
\draw [line width=1.2pt] (5.998362179737032,3.6127136855130266)-- (0.,2.0008495480853425);
\draw [line width=1.2pt] (0.,2.0008495480853425)-- (0.,6.);
\draw [line width=1.2pt] (0.,6.)-- (6.0109548683106855,8.435713409222425);
\begin{scriptsize}
\draw [fill=black] (0.,2.0008495480853425) circle (4.5pt);
\draw [fill=black] (5.998362179737032,3.6127136855130266) circle (4.5pt);
\draw [fill=black] (5.003539782418382,0.6282464935570807) circle (4.5pt);
\draw [fill=black] (0.,6.) circle (4.5pt);
\draw [fill=black] (6.0109548683106855,8.435713409222425) circle (4.5pt);
\draw [fill=black] (5.,5.) circle (4.5pt);
\end{scriptsize}
\end{tikzpicture}
\end{subfigure} \quad
\begin{subfigure}[b]{0.2\textwidth}
\tikzset{every picture/.style={scale=0.3}}%
\definecolor{ccccff}{rgb}{0.8,0.8,1.}
\begin{tikzpicture}[line cap=round,line join=round,>=triangle 45,x=1.0cm,y=1.0cm]
\clip(-2.1617000159906343,-0.819912692413104) rectangle (19.019202164895066,9.266830855083574);
\fill[line width=2.pt,color=ccccff,fill=ccccff,fill opacity=0.10000000149011612] (0.,7.) -- (2.,9.) -- (6.,8.) -- (4.,6.) -- cycle;
\fill[line width=2.pt,color=ccccff,fill=ccccff,fill opacity=0.10000000149011612] (0.,7.) -- (2.,9.) -- (2.,4.) -- (0.,2.0008495480853425) -- cycle;
\fill[line width=2.pt,color=ccccff,fill=ccccff,fill opacity=0.10000000149011612] (2.,9.) -- (6.,8.) -- (6.,3.) -- (2.,4.) -- cycle;
\fill[line width=2.pt,color=ccccff,fill=ccccff,fill opacity=0.10000000149011612] (6.,8.) -- (4.,6.) -- (4.,1.) -- (6.,3.) -- cycle;
\fill[line width=2.pt,color=ccccff,fill=ccccff,fill opacity=0.10000000149011612] (0.,7.) -- (4.,6.) -- (4.,1.) -- (0.,2.0008495480853425) -- cycle;
\fill[line width=2.pt,color=ccccff,fill=ccccff,fill opacity=0.10000000149011612] (2.,4.) -- (6.,3.) -- (4.,1.) -- (0.,2.0008495480853425) -- cycle;
\draw [line width=1.2pt] (0.,7.)-- (2.,9.);
\draw [line width=1.2pt] (2.,9.)-- (6.,8.);
\draw [line width=1.2pt] (6.,8.)-- (4.,6.);
\draw [line width=1.2pt] (4.,6.)-- (0.,7.);
\draw [line width=1.2pt] (0.,7.)-- (2.,9.);
\draw [line width=1.2pt] (2.,9.)-- (2.,4.);
\draw [line width=1.2pt] (2.,4.)-- (0.,2.0008495480853425);
\draw [line width=1.2pt] (0.,2.0008495480853425)-- (0.,7.);
\draw [line width=1.2pt] (2.,9.)-- (6.,8.);
\draw [line width=1.2pt] (6.,8.)-- (6.,3.);
\draw [line width=1.2pt] (6.,3.)-- (2.,4.);
\draw [line width=1.2pt] (2.,4.)-- (2.,9.);
\draw [line width=1.2pt] (6.,8.)-- (4.,6.);
\draw [line width=1.2pt] (4.,6.)-- (4.,1.);
\draw [line width=1.2pt] (4.,1.)-- (6.,3.);
\draw [line width=1.2pt] (6.,3.)-- (6.,8.);
\draw [line width=1.2pt] (0.,7.)-- (4.,6.);
\draw [line width=1.2pt] (4.,6.)-- (4.,1.);
\draw [line width=1.2pt] (4.,1.)-- (0.,2.0008495480853425);
\draw [line width=1.2pt] (0.,2.0008495480853425)-- (0.,7.);
\draw [line width=1.2pt] (2.,4.)-- (6.,3.);
\draw [line width=1.2pt] (6.,3.)-- (4.,1.);
\draw [line width=1.2pt] (4.,1.)-- (0.,2.0008495480853425);
\draw [line width=1.2pt] (0.,2.0008495480853425)-- (2.,4.);
\begin{scriptsize}
\draw [fill=black] (0.,2.0008495480853425) circle (4.5pt);
\draw [fill=black] (6.,3.) circle (4.5pt);
\draw [fill=black] (4.,1.) circle (4.5pt);
\draw [fill=black] (2.,4.) circle (4.5pt);
\draw [fill=black] (0.,7.) circle (4.5pt);
\draw [fill=black] (2.,9.) circle (4.5pt);
\draw [fill=black] (6.,8.) circle (4.5pt);
\draw [fill=black] (4.,6.) circle (4.5pt);
\end{scriptsize}
\end{tikzpicture}
\end{subfigure}

\vspace{-0.43in}

\begin{subfigure}[b]{0.2\textwidth}
\tikzset{every picture/.style={scale=0.3}}%
\definecolor{ccccff}{rgb}{0.8,0.8,1.}
\begin{tikzpicture}[line cap=round,line join=round,>=triangle 45,x=1.0cm,y=1.0cm]
\clip(-0.3973712171493665,-2.08163848033795) rectangle (28.591996026839478,11.723642829076395);
\fill[line width=2.pt,color=ccccff,fill=ccccff,fill opacity=0.10000000149011612] (-0.041091260187331294,5.9952503636483145) -- (0.,2.0008495480853425) -- (2.3947081487996185,5.00197028142147) -- (2.411943206614118,9.017738752199925) -- cycle;
\fill[line width=2.pt,color=ccccff,fill=ccccff,fill opacity=0.10000000149011612] (6.0313053476590985,8.000870341144436) -- (5.,4.69658646900089) -- (5.,1.) -- (6.,4.) -- cycle;
\fill[line width=2.pt,color=ccccff,fill=ccccff,fill opacity=0.10000000149011612] (2.,3.6965864690008896) -- (2.,0.) -- (5.,1.) -- (5.,4.69658646900089) -- cycle;
\fill[line width=2.pt,color=ccccff,fill=ccccff,fill opacity=0.10000000149011612] (-0.041091260187331294,5.9952503636483145) -- (0.,2.0008495480853425) -- (2.,0.) -- (2.,3.6965864690008896) -- cycle;
\fill[line width=2.pt,color=ccccff,fill=ccccff,fill opacity=0.10000000149011612] (2.411943206614118,9.017738752199925) -- (6.0313053476590985,8.000870341144436) -- (6.,4.) -- (2.3947081487996185,5.00197028142147) -- cycle;
\draw [line width=1.2pt] (-0.041091260187331294,5.9952503636483145)-- (0.,2.0008495480853425);
\draw [line width=1.2pt] (0.,2.0008495480853425)-- (2.3947081487996185,5.00197028142147);
\draw [line width=1.2pt] (2.3947081487996185,5.00197028142147)-- (2.411943206614118,9.017738752199925);
\draw [line width=1.2pt] (2.411943206614118,9.017738752199925)-- (-0.041091260187331294,5.9952503636483145);
\draw [line width=1.2pt] (6.0313053476590985,8.000870341144436)-- (5.,4.69658646900089);
\draw [line width=1.2pt] (5.,4.69658646900089)-- (5.,1.);
\draw [line width=1.2pt] (5.,1.)-- (6.,4.);
\draw [line width=1.2pt] (6.,4.)-- (6.0313053476590985,8.000870341144436);
\draw [line width=1.2pt] (2.,3.6965864690008896)-- (2.,0.);
\draw [line width=1.2pt] (2.,0.)-- (5.,1.);
\draw [line width=1.2pt] (5.,1.)-- (5.,4.69658646900089);
\draw [line width=1.2pt] (5.,4.69658646900089)-- (2.,3.6965864690008896);
\draw [line width=1.2pt] (-0.041091260187331294,5.9952503636483145)-- (0.,2.0008495480853425);
\draw [line width=1.2pt] (0.,2.0008495480853425)-- (2.,0.);
\draw [line width=1.2pt] (2.,0.)-- (2.,3.6965864690008896);
\draw [line width=1.2pt] (2.,3.6965864690008896)-- (-0.041091260187331294,5.9952503636483145);
\draw [line width=1.6pt] (2.3947081487996185,5.00197028142147)-- (6.,4.);
\draw [line width=1.6pt] (2.411943206614118,9.017738752199925)-- (6.0313053476590985,8.000870341144436);
\draw [line width=2.pt,color=ccccff] (2.411943206614118,9.017738752199925)-- (6.0313053476590985,8.000870341144436);
\draw [line width=2.pt,color=ccccff] (6.0313053476590985,8.000870341144436)-- (6.,4.);
\draw [line width=2.pt,color=ccccff] (6.,4.)-- (2.3947081487996185,5.00197028142147);
\draw [line width=2.pt,color=ccccff] (2.3947081487996185,5.00197028142147)-- (2.411943206614118,9.017738752199925);
\draw [line width=1.6pt] (2.411943206614118,9.017738752199925)-- (6.0313053476590985,8.000870341144436);
\draw [line width=1.6pt] (2.411943206614118,9.017738752199925)-- (2.3947081487996185,5.00197028142147);
\draw [line width=1.6pt] (2.3947081487996185,5.00197028142147)-- (6.,4.);
\draw [line width=1.6pt] (6.,4.)-- (6.0313053476590985,8.000870341144436);
\begin{scriptsize}
\draw [fill=black] (0.,2.0008495480853425) circle (4.5pt);
\draw [fill=black] (6.,4.) circle (4.5pt);
\draw [fill=black] (2.,0.) circle (4.5pt);
\draw [fill=black] (5.,1.) circle (4.5pt);
\draw [fill=black] (2.3947081487996185,5.00197028142147) circle (4.5pt);
\draw [fill=black] (-0.041091260187331294,5.9952503636483145) circle (4.5pt);
\draw [fill=black] (2.411943206614118,9.017738752199925) circle (4.5pt);
\draw [fill=black] (2.,3.6965864690008896) circle (4.5pt);
\draw [fill=black] (5.,4.69658646900089) circle (4.5pt);
\draw [fill=black] (6.0313053476590985,8.000870341144436) circle (4.5pt);
\end{scriptsize}
\end{tikzpicture}
\end{subfigure} \quad
\begin{subfigure}[b]{0.2\textwidth}
\tikzset{every picture/.style={scale=0.3}}%
\definecolor{ccccff}{rgb}{0.8,0.8,1.}
\begin{tikzpicture}[line cap=round,line join=round,>=triangle 45,x=1.0cm,y=1.0cm]
\clip(-0.6903671999958648,-2.460809752256946) rectangle (28.678171315911975,11.723642829076395);
\fill[line width=2.pt,color=ccccff,fill=ccccff,fill opacity=0.10000000149011612] (-0.041091260187331294,5.9952503636483145) -- (0.9562496748460502,8.00378419114609) -- (3.9151546828520476,9.021915422787718) -- (5.989075352833551,7.50635801010893) -- (5.,4.69658646900089) -- (2.,3.6965864690008896) -- cycle;
\fill[line width=2.pt,color=ccccff,fill=ccccff,fill opacity=0.10000000149011612] (-0.041091260187331294,5.9952503636483145) -- (0.,2.0008495480853425) -- (0.9861480425689808,4.540226339809746) -- (0.9562496748460502,8.00378419114609) -- cycle;
\fill[line width=2.pt,color=ccccff,fill=ccccff,fill opacity=0.10000000149011612] (0.9562496748460502,8.00378419114609) -- (3.9151546828520476,9.021915422787718) -- (3.9861480425689813,5.540226339809746) -- (0.9861480425689808,4.540226339809746) -- cycle;
\fill[line width=2.pt,color=ccccff,fill=ccccff,fill opacity=0.10000000149011612] (3.9151546828520476,9.021915422787718) -- (5.989075352833551,7.50635801010893) -- (6.,4.) -- (3.9861480425689813,5.540226339809746) -- cycle;
\fill[line width=2.pt,color=ccccff,fill=ccccff,fill opacity=0.10000000149011612] (5.989075352833551,7.50635801010893) -- (5.,4.69658646900089) -- (5.,1.) -- (6.,4.) -- cycle;
\fill[line width=2.pt,color=ccccff,fill=ccccff,fill opacity=0.10000000149011612] (2.,3.6965864690008896) -- (2.,0.) -- (5.,1.) -- (5.,4.69658646900089) -- cycle;
\fill[line width=2.pt,color=ccccff,fill=ccccff,fill opacity=0.10000000149011612] (-0.041091260187331294,5.9952503636483145) -- (0.,2.0008495480853425) -- (2.,0.) -- (2.,3.6965864690008896) -- cycle;
\fill[line width=2.pt,color=ccccff,fill=ccccff,fill opacity=0.10000000149011612] (0.,2.0008495480853425) -- (0.9861480425689808,4.540226339809746) -- (3.9861480425689813,5.540226339809746) -- (6.,4.) -- (5.,1.) -- (2.,0.) -- cycle;
\draw [line width=1.2pt] (-0.041091260187331294,5.9952503636483145)-- (0.9562496748460502,8.00378419114609);
\draw [line width=1.2pt] (0.9562496748460502,8.00378419114609)-- (3.9151546828520476,9.021915422787718);
\draw [line width=1.2pt] (3.9151546828520476,9.021915422787718)-- (5.989075352833551,7.50635801010893);
\draw [line width=1.2pt] (5.989075352833551,7.50635801010893)-- (5.,4.69658646900089);
\draw [line width=1.2pt] (5.,4.69658646900089)-- (2.,3.6965864690008896);
\draw [line width=1.2pt] (2.,3.6965864690008896)-- (-0.041091260187331294,5.9952503636483145);
\draw [line width=1.2pt] (-0.041091260187331294,5.9952503636483145)-- (0.,2.0008495480853425);
\draw [line width=1.2pt] (0.,2.0008495480853425)-- (0.9861480425689808,4.540226339809746);
\draw [line width=1.2pt] (0.9861480425689808,4.540226339809746)-- (0.9562496748460502,8.00378419114609);
\draw [line width=1.2pt] (0.9562496748460502,8.00378419114609)-- (-0.041091260187331294,5.9952503636483145);
\draw [line width=1.2pt] (0.9562496748460502,8.00378419114609)-- (3.9151546828520476,9.021915422787718);
\draw [line width=1.2pt] (3.9151546828520476,9.021915422787718)-- (3.9861480425689813,5.540226339809746);
\draw [line width=1.2pt] (3.9861480425689813,5.540226339809746)-- (0.9861480425689808,4.540226339809746);
\draw [line width=1.2pt] (0.9861480425689808,4.540226339809746)-- (0.9562496748460502,8.00378419114609);
\draw [line width=1.2pt] (3.9151546828520476,9.021915422787718)-- (5.989075352833551,7.50635801010893);
\draw [line width=1.2pt] (5.989075352833551,7.50635801010893)-- (6.,4.);
\draw [line width=1.2pt] (6.,4.)-- (3.9861480425689813,5.540226339809746);
\draw [line width=1.2pt] (3.9861480425689813,5.540226339809746)-- (3.9151546828520476,9.021915422787718);
\draw [line width=1.2pt] (5.989075352833551,7.50635801010893)-- (5.,4.69658646900089);
\draw [line width=1.2pt] (5.,4.69658646900089)-- (5.,1.);
\draw [line width=1.2pt] (5.,1.)-- (6.,4.);
\draw [line width=1.2pt] (6.,4.)-- (5.989075352833551,7.50635801010893);
\draw [line width=1.2pt] (2.,3.6965864690008896)-- (2.,0.);
\draw [line width=1.2pt] (2.,0.)-- (5.,1.);
\draw [line width=1.2pt] (5.,1.)-- (5.,4.69658646900089);
\draw [line width=1.2pt] (5.,4.69658646900089)-- (2.,3.6965864690008896);
\draw [line width=1.2pt] (-0.041091260187331294,5.9952503636483145)-- (0.,2.0008495480853425);
\draw [line width=1.2pt] (0.,2.0008495480853425)-- (2.,0.);
\draw [line width=1.2pt] (2.,0.)-- (2.,3.6965864690008896);
\draw [line width=1.2pt] (2.,3.6965864690008896)-- (-0.041091260187331294,5.9952503636483145);
\draw [line width=1.2pt] (0.,2.0008495480853425)-- (0.9861480425689808,4.540226339809746);
\draw [line width=1.2pt] (0.9861480425689808,4.540226339809746)-- (3.9861480425689813,5.540226339809746);
\draw [line width=1.2pt] (3.9861480425689813,5.540226339809746)-- (6.,4.);
\draw [line width=1.2pt] (6.,4.)-- (5.,1.);
\draw [line width=1.2pt] (5.,1.)-- (2.,0.);
\draw [line width=1.2pt] (2.,0.)-- (0.,2.0008495480853425);
\begin{scriptsize}
\draw [fill=black] (0.,2.0008495480853425) circle (4.5pt);
\draw [fill=black] (6.,4.) circle (4.5pt);
\draw [fill=black] (2.,0.) circle (4.5pt);
\draw [fill=black] (5.,1.) circle (4.5pt);
\draw [fill=black] (0.9861480425689808,4.540226339809746) circle (4.5pt);
\draw [fill=black] (3.9861480425689813,5.540226339809746) circle (4.5pt);
\draw [fill=black] (-0.041091260187331294,5.9952503636483145) circle (4.5pt);
\draw [fill=black] (0.9562496748460502,8.00378419114609) circle (4.5pt);
\draw [fill=black] (3.9151546828520476,9.021915422787718) circle (4.5pt);
\draw [fill=black] (2.,3.6965864690008896) circle (4.5pt);
\draw [fill=black] (5.,4.69658646900089) circle (4.5pt);
\draw [fill=black] (5.989075352833551,7.50635801010893) circle (4.5pt);
\end{scriptsize}
\end{tikzpicture}
\end{subfigure} 
\begin{subfigure}[b]{0.2\textwidth}
\tikzset{every picture/.style={scale=0.3}}%
\definecolor{ccccff}{rgb}{0.8,0.8,1.}
\begin{tikzpicture}[line cap=round,line join=round,>=triangle 45,x=1.0cm,y=1.0cm]
\clip(-0.8701430064260149,-1.6639407310128818) rectangle (24.171261036089607,10.231470582881203);
\fill[line width=2.pt,color=ccccff,fill=ccccff,fill opacity=0.10000000149011612] (0.1886207485436303,4.795149587647471) -- (4.342269681838108,4.571093836058975) -- (5.652134075740101,3.7093409453339845) -- (6.841353064940594,2.795882881165495) -- (7.7945996527559664,2.013364143244739) -- (7.928590518619606,1.0754280821992728) -- (7.422402803134747,0.) -- cycle;
\fill[line width=0.pt,color=ccccff,fill=ccccff,fill opacity=0.10000000149011612] (7.422402803134747,0.) -- (8.83675083169538,0.4352495008507801) -- (7.928590518619606,1.0754280821992728) -- cycle;
\fill[line width=0.pt,color=ccccff,fill=ccccff,fill opacity=0.10000000149011612] (6.841353064940594,2.795882881165495) -- (8.36033886418022,3.159730440078086) -- (7.7945996527559664,2.013364143244739) -- cycle;
\fill[line width=0.pt,color=ccccff,fill=ccccff,fill opacity=0.10000000149011612] (5.652134075740101,3.7093409453339845) -- (9.402490043119633,4.425199728790223) -- (8.36033886418022,3.159730440078086) -- (6.841353064940594,2.795882881165495) -- cycle;
\fill[line width=0.pt,color=ccccff,fill=ccccff,fill opacity=0.10000000149011612] (4.342269681838108,4.571093836058975) -- (10.399977600104501,5.690669017502359) -- (9.402490043119633,4.425199728790223) -- (5.652134075740101,3.7093409453339845) -- cycle;
\fill[line width=0.pt,color=ccccff,fill=ccccff,fill opacity=0.10000000149011612] (0.1886207485436303,4.795149587647471) -- (13.600870506846984,7.5814290135781395) -- (10.399977600104501,5.690669017502359) -- (4.342269681838108,4.571093836058975) -- cycle;
\fill[line width=2.pt,color=ccccff,fill=ccccff,fill opacity=0.10000000149011612] (13.600870506846984,7.5814290135781395) -- (10.399977600104501,5.690669017502359) -- (9.402490043119633,4.425199728790223) -- (8.36033886418022,3.159730440078086) -- (7.7945996527559664,2.013364143244739) -- (7.928590518619606,1.0754280821992728) -- (8.83675083169538,0.4352495008507801) -- cycle;
\draw [line width=1.2pt,color=ccccff] (0.1886207485436303,4.795149587647471)-- (4.342269681838108,4.571093836058975);
\draw [line width=1.2pt,color=ccccff] (4.342269681838108,4.571093836058975)-- (5.652134075740101,3.7093409453339845);
\draw [line width=1.2pt,color=ccccff] (5.652134075740101,3.7093409453339845)-- (6.841353064940594,2.795882881165495);
\draw [line width=1.2pt,color=ccccff] (6.841353064940594,2.795882881165495)-- (7.7945996527559664,2.013364143244739);
\draw [line width=1.2pt,color=ccccff] (7.7945996527559664,2.013364143244739)-- (7.928590518619606,1.0754280821992728);
\draw [line width=1.2pt,color=ccccff] (7.928590518619606,1.0754280821992728)-- (7.422402803134747,0.);
\draw [line width=1.2pt,color=ccccff] (7.422402803134747,0.)-- (0.1886207485436303,4.795149587647471);
\draw [line width=1.2pt,color=ccccff] (7.422402803134747,0.)-- (8.83675083169538,0.4352495008507801);
\draw [line width=1.2pt,color=ccccff] (8.83675083169538,0.4352495008507801)-- (7.928590518619606,1.0754280821992728);
\draw [line width=1.2pt,color=ccccff] (7.928590518619606,1.0754280821992728)-- (7.422402803134747,0.);
\draw [line width=1.2pt,color=ccccff] (6.841353064940594,2.795882881165495)-- (8.36033886418022,3.159730440078086);
\draw [line width=1.2pt,color=ccccff] (8.36033886418022,3.159730440078086)-- (7.7945996527559664,2.013364143244739);
\draw [line width=1.2pt,color=ccccff] (7.7945996527559664,2.013364143244739)-- (6.841353064940594,2.795882881165495);
\draw [line width=1.2pt,color=ccccff] (5.652134075740101,3.7093409453339845)-- (9.402490043119633,4.425199728790223);
\draw [line width=1.2pt,color=ccccff] (9.402490043119633,4.425199728790223)-- (8.36033886418022,3.159730440078086);
\draw [line width=1.2pt,color=ccccff] (8.36033886418022,3.159730440078086)-- (6.841353064940594,2.795882881165495);
\draw [line width=1.2pt,color=ccccff] (6.841353064940594,2.795882881165495)-- (5.652134075740101,3.7093409453339845);
\draw [line width=1.2pt,color=ccccff] (4.342269681838108,4.571093836058975)-- (10.399977600104501,5.690669017502359);
\draw [line width=1.2pt,color=ccccff] (10.399977600104501,5.690669017502359)-- (9.402490043119633,4.425199728790223);
\draw [line width=1.2pt,color=ccccff] (9.402490043119633,4.425199728790223)-- (5.652134075740101,3.7093409453339845);
\draw [line width=1.2pt,color=ccccff] (5.652134075740101,3.7093409453339845)-- (4.342269681838108,4.571093836058975);
\draw [line width=1.2pt,color=ccccff] (0.1886207485436303,4.795149587647471)-- (13.600870506846984,7.5814290135781395);
\draw [line width=1.2pt,color=ccccff] (13.600870506846984,7.5814290135781395)-- (10.399977600104501,5.690669017502359);
\draw [line width=1.2pt,color=ccccff] (10.399977600104501,5.690669017502359)-- (4.342269681838108,4.571093836058975);
\draw [line width=1.2pt,color=ccccff] (4.342269681838108,4.571093836058975)-- (0.1886207485436303,4.795149587647471);
\draw [line width=1.2pt,color=ccccff] (13.600870506846984,7.5814290135781395)-- (10.399977600104501,5.690669017502359);
\draw [line width=1.2pt,color=ccccff] (10.399977600104501,5.690669017502359)-- (9.402490043119633,4.425199728790223);
\draw [line width=1.2pt,color=ccccff] (9.402490043119633,4.425199728790223)-- (8.36033886418022,3.159730440078086);
\draw [line width=1.2pt,color=ccccff] (8.36033886418022,3.159730440078086)-- (7.7945996527559664,2.013364143244739);
\draw [line width=1.2pt,color=ccccff] (7.7945996527559664,2.013364143244739)-- (7.928590518619606,1.0754280821992728);
\draw [line width=1.2pt,color=ccccff] (7.928590518619606,1.0754280821992728)-- (8.83675083169538,0.4352495008507801);
\draw [line width=1.2pt,color=ccccff] (8.83675083169538,0.4352495008507801)-- (13.600870506846984,7.5814290135781395);
\draw [line width=1.2pt] (0.1886207485436303,4.795149587647471)-- (4.342269681838108,4.571093836058975);
\draw [line width=1.2pt] (4.342269681838108,4.571093836058975)-- (5.652134075740101,3.7093409453339845);
\draw [line width=1.2pt] (5.652134075740101,3.7093409453339845)-- (6.841353064940594,2.795882881165495);
\draw [line width=1.2pt] (6.841353064940594,2.795882881165495)-- (7.7945996527559664,2.013364143244739);
\draw [line width=1.2pt] (7.7945996527559664,2.013364143244739)-- (7.928590518619606,1.0754280821992728);
\draw [line width=1.2pt] (7.928590518619606,1.0754280821992728)-- (7.422402803134747,0.);
\draw [line width=1.2pt] (7.422402803134747,0.)-- (0.1886207485436303,4.795149587647471);
\draw [line width=1.2pt] (0.1886207485436303,4.795149587647471)-- (13.600870506846984,7.5814290135781395);
\draw [line width=1.2pt] (13.600870506846984,7.5814290135781395)-- (10.399977600104501,5.690669017502359);
\draw [line width=1.2pt] (10.399977600104501,5.690669017502359)-- (4.342269681838108,4.571093836058975);
\draw [line width=1.2pt] (5.652134075740101,3.7093409453339845)-- (9.402490043119633,4.425199728790223);
\draw [line width=1.2pt] (9.402490043119633,4.425199728790223)-- (10.399977600104501,5.690669017502359);
\draw [line width=1.2pt] (6.841353064940594,2.795882881165495)-- (8.36033886418022,3.159730440078086);
\draw [line width=1.2pt] (9.402490043119633,4.425199728790223)-- (8.36033886418022,3.159730440078086);
\draw [line width=1.2pt] (8.36033886418022,3.159730440078086)-- (7.7945996527559664,2.013364143244739);
\draw [line width=1.2pt] (7.422402803134747,0.)-- (8.83675083169538,0.4352495008507801);
\draw [line width=1.2pt] (8.83675083169538,0.4352495008507801)-- (7.928590518619606,1.0754280821992728);
\draw [line width=1.2pt] (8.83675083169538,0.4352495008507801)-- (13.600870506846984,7.5814290135781395);
\begin{scriptsize}
\draw [fill=black] (0.1886207485436303,4.795149587647471) circle (4.5pt);
\draw [fill=black] (8.83675083169538,0.4352495008507801) circle (4.5pt);
\draw [fill=black] (7.7945996527559664,2.013364143244739) circle (4.5pt);
\draw [fill=black] (6.841353064940594,2.795882881165495) circle (4.5pt);
\draw [fill=black] (8.36033886418022,3.159730440078086) circle (4.5pt);
\draw [fill=black] (5.652134075740101,3.7093409453339845) circle (4.5pt);
\draw [fill=black] (9.402490043119633,4.425199728790223) circle (4.5pt);
\draw [fill=black] (4.342269681838108,4.571093836058975) circle (4.5pt);
\draw [fill=black] (10.399977600104501,5.690669017502359) circle (4.5pt);
\draw [fill=black] (13.600870506846984,7.5814290135781395) circle (4.5pt);
\draw [fill=black] (7.422402803134747,0.) circle (4.5pt);
\draw [fill=black] (7.928590518619606,1.0754280821992728) circle (4.5pt);
\end{scriptsize}
\end{tikzpicture}
\end{subfigure}
\vspace{-0.3in}
\caption{K3 polytopes obtained from triangulations of minimal polytopes. \label{fig:K3minimal}}
\end{figure}

\section{Stability of quartic surfaces} 

In this section we prove Theorem \ref{thm:2}. By Corollary \ref{cor:stability},
it suffices to consider quartics $f$ such that $P = {\rm Newt}(f)$ is 
one of the $115$ minimal polytopes in Proposition~\ref{prop:minimal}.
We write $f$ as a homogeneous polynomial in four variables $x,y,z,w$.
The monomials in $f$ correspond to the points in $P \cap \mathbb{Z}^4 = \{m_1,m_2,\ldots,m_r\}$.
One of these is ${\bf p} = (1,1,1,1)$.
The $r$ monomials span a linear system inside
$\mathbb{HS}_{4,3} = H^0\big(\mathbb{P}^3, \mathcal{O}_{\mathbb{P}^3}(4)\big)$,
and we assume that $f$ is a general element of this linear system.
By Bertini's Theorem, the surface $V(f)$ is smooth outside the base locus
$V(m_1,m_2,\ldots,m_r)$ in $\PP^3$.

We begin with the following remark.
Suppose $u$ is a point of the base locus which is singular of multiplicity at least $s+1$ in all 
divisors $V(m_i)$, and $u$ has multiplicity exactly $s+1$ for at least one divisor. Then $u$ is a singular point of multiplicity exactly $s+1$ for the general element of the linear system. More precisely, let $u \in \mathbb{P}^3$ such that $u \in V(m_i)$ for every $i \in \{1, 2, \dots, r\}$ and such that 
\[ \frac{\partial^{|\alpha|} m_i}{\partial^{\alpha} x} (u) = 0 \ \ \text{for every multindex $\alpha$ such that $|\alpha| \leq s$, and every $i \in \{1, 2, \dots, r\}$}, \]
\[ \frac{\partial^{|\alpha|} m_j}{\partial^{\alpha} x} (u) \not = 0 \ \ \text{for at least one multindex $\alpha$ with $|\alpha| = s+1$, and $j \in \{1, 2, \dots, r\}$}. \]
The set of all $(\lambda_1, \lambda_2, \dots, \lambda_r) \in \mathbb{C}^r$ such that the point $u$ is singular of multiplicity $s+1$ in the surface $f = \lambda_1 m_1 + \dots + \lambda_r m_r$ is a Zariski open dense subset of $\mathbb{C}^r$. 

In what follows we establish the stability of quartics whose Newton polytope is minimal. For achieving this we will not use Theorem~\ref{thm:mumford}, for which a condition on the Newton polytope needs to be checked for an arbitrary change of coordinates. Instead,
we implement a computer-assisted verification capable of dealing with polynomials with general coefficients. Specifically, we first compute the singular points of the base locus defined by the monomials $m_1,\ldots,m_r$ 
of each minimal polytope. This does not depend on the choice of coefficients, which are only assumed to be nonzero. Then we move the coordinates to those of an affine neighborhood chosen such that the singular point is the origin. There, we perform changes of coordinates of the polynomials to reduce the singularity to a normal form and compare it with the ones characterized by Arnol\cprime d in \cite{Arn72}. These changes of coordinates are at worst polynomial, and finite in number. We keep track of what happens to the coefficient of each monomial during this procedure to make sure that no cancellation takes place. The finiteness of this process preserves the genericity assumption on the polynomial. We use this to deduce the stability of generic polynomials arising from minimal polytopes, and, consequently, from all the $15 \; 139$ reflexive polytopes of Corollary~\ref{cor:reflexive}. This is done working with polynomials having general coefficients, so the stability is checked in a dense open subset of the space of coefficients.

Our main result in this section is Algorithm \ref{al:singularities}. We implemented this algorithm in {\tt Python}. Our code takes care of performing all the calculations and manipulating
polynomials having one of the $115$ minimal polytopes as Newton polytope. Our source code is 
available on {\tt GitHub} at \url{https://github.com/gabrieleballetti/singularities}. 

In order to work with general coefficients, the coefficients of the monomials are defined 
to be the variables of a new polynomial ring. For a fixed minimal polytope $P$, let
$f = \lambda_1 m_1 + \cdots + \lambda_r m_r$ be the general polynomial
with $P \cap \ZZ^4 = \{m_1,\ldots,m_r\}$. Our script regards $f$ as an element of
$\QQ[\lambda_1,\ldots,\lambda_r][x,y,z,w]$.
Using $\QQ$ as a base field is sufficient as all manipulations we perform involve rational numbers. 
The coefficient vectors can be thought throughout as general elements in $(\CC^*)^r$.
 For some manipulations, we might require that the coefficients of some monomials do not vanish.
 These will be expressions in $\QQ[\lambda_1,\ldots,\lambda_r]$,
 so our results are valid over a Zariski dense subset in $\CC^r$.
 
For each minimal polytope $P$, we compute
the singular points of the base locus $V(m_1,\ldots,m_r)$.
We find that each singular point is a coordinate point.
Using the initial remark above, we conclude that
this is also the singular locus of the surface $V(f)$ where
$f = \lambda_1 m_1 + \cdots + \lambda_r m_r$ is generic.

\begin{proposition}
Given a minimal polytope $P$, the general surface with  Newton polytope $P$ has 
isolated singularities.  All singular points are coordinate points and they have multiplicity two.
\end{proposition}
 
In order to conclude the stability of the surfaces, this information is not enough.
We need to understand whether they are rational double points or not, according to
Shah's classification \cite{Shah}. 

In what follows we shall use the classification of hypersurface singularities due to
Arnol\cprime d \cite{Arn72, Arn74}.  We also refer to the monograph
by Greuel et al.~\cite{GLS}. For our quartic surfaces, the singular points are
\begin{equation}
\label{eq:singularpoints}
[1,0,0,0], \; [0,1,0,0] , \; [0,0,1,0], \; \text{or} \; [0,0,0,1]  \,\,\, \, {\rm in} \, \,\, \PP^3. 
\end{equation}
When analyzing a singular point $u$,
 we always work in an affine neighborhood, where $u$ is the origin.
 We do this by dehomogenizing $f$ with respect to     $x, y, z$ or $w$.
 After relabeling the variables, we regard our polynomials as elements 
 in the ring $R = \mathbb{C}\{x, y, z\}$ of convergent power series.  This ring is local,
 with maximal ideal  $   \mathfrak{m} = \langle x,y,z \rangle$.
 If $f \in \mathfrak{m}^r$, then $f$ has a zero of order $r$ at the origin.

\begin{definition}
Let $f,g \in R$. We define the \emph{$k$-jet} $J_k \, f$ as the image of $f$ in ${R}/{\mathfrak{m}^k}$.
We say that $f$ is \emph{right equivalent} to $g$, denoted $f \sim g$, if there exists an automorphism $\varphi$ of $R$ such that $\varphi(f) =g$. We say that $f$ is
 \emph{right $k$-determined} if  $f \sim h$ for each $h \in R$ with $J_k \, f = J_k \, h$.
\end{definition}

The following classical result gives a sufficient condition for 
a singularity to be $k$-determined.  For a reference
see \cite[Theorem 2.23]{GLS} or \cite[Lemma 3.1--3.2]{Arn72}.

\begin{theorem}[Arnol\cprime d's Finite Determinacy Theorem] Let $f \in \mathfrak{m}$. Then $f$ is right $k$-determined~if 
\[\mathfrak{m}^{k+1}\,\, \subseteq \,\,\mathfrak{m}^2 \cdot \Big \langle \frac{\partial f}{\partial x}, \frac{\partial f}{\partial y}, \frac{\partial f}{\partial z}\Big\rangle. \]
\end{theorem}

Arnol\cprime d \cite{Arn72} classifies the normal forms of a function in the neighborhood of a simple critical point:

\begin{theorem}[Arnol\cprime d] 
\label{thm:class_and} 
Each rational double point is right equivalent to one of the following normal~forms: \hfill
(These are indexed by root systems, and we refer to them in Table \ref{tab:sing}.)
\begin{itemize}
\item[$A_k$:] $x^2+y^2+ z^{k+1}$, 
\item[$D_k$:] $x^2 + y^2z + z^{k-1}$, 
\item[$E_6$:] $x^2 + y^3 + z^4$, 
\item[$E_7$:] $x^2 + y^3 + yz^3$, 
\item[$E_8$:] $x^2 + y^3 + z^5$. 
\end{itemize}
The normal forms $A_k$, $D_k$, $E_6$, $E_7$ and $E_8$ are respectively $k+1$, $k-1$, $4$, $4$ and $5$-determined.
\end{theorem}

We prove Theorem~\ref{thm:2} by checking that each 
quartic $f$ with general coefficients $\lambda_1,\ldots,\lambda_r$ as above,
with Newton polytope from Proposition \ref{prop:minimal},
is right equivalent to one the forms listed above.

\smallskip

We next recall the \emph{splitting lemma}, also known as \emph{generalized Morse lemma};
see \cite[Lemma 4.1]{Arn72} and \cite[Theorem 2.47]{GLS}. We quickly sketch 
the proof given in \cite{GLS}, as the method will be  essential in our algorithm for determining  stability. 

\begin{theorem}[Arnol\cprime d's Splitting Lemma] 
Let $f \in \mathfrak{m}^2 \subset \mathbb{C}\{x_1, x_2, \dots, x_n\}$. If the Hessian matrix $H(f)(0)$ at the point $0$ has rank $k$, then 
\[
f  \,\,\sim \,\, x_1^2 + x_2^2 + \cdots +x_k^2 + g(x_{k+1}, \dots, x_n), 
\]
with $g \in \mathfrak{m}^3$, uniquely determined up to right equivalence. 
\end{theorem}
\begin{proof}
By Jacobi's Theorem, we can assume that the $2$-jet of $f$ is 
$\, J_2 \, f = x_1^2 + x_2^2 + \cdots +x_k^2$. Hence
\[
f \,\,=\,\, x_1^2 + x_2^2 + \cdots +x_k^2 \,+\, f_{\geq 3}(x_{k+1}, \dots, x_n) \,+\, \sum_{i=1}^k x_i \, g_i(x_1, x_2, \dots, x_n),
\]
with $f_{\geq 3} \in \mathfrak{m}^3$ and $g_i \in \mathfrak{m}^2$. 
We apply the following change of coordinates:
\begin{equation}
\label{eq:transformation}
\begin{array}{ccccc}
 x_i & \mapsto & x_i - \frac{1}{2}g_i(x_1, x_2, \dots, x_n) \quad && \textrm{for} \ 1\leq i \leq k, \\
 x_i & \mapsto  & x_i \qquad \qquad && \textrm{for} \ i>k.
\end{array}
\end{equation}
We get
\[
f \,\,=\,\, x_1^2 + x_2^2 + \cdots +x_k^2\, + \,
f_{\geq 3}(x_{k+1}, \dots, x_n) \,+\,
 f_{\geq 4}(x_{k+1}, \dots, x_n) \,+ \, \sum_{i=1}^k x_i \, h_i(x_1, x_2, \dots, x_n),\]
with $f_{\geq 4} \in \mathfrak{m}^4$ and $h_i \in \mathfrak{m}^3$. The statement follows by recursively repeating the same argument. 
\end{proof}

Algorithm \ref{al:singularities} is used to to classify the singularities of our surfaces. It begins by looking at the $2$-jets of the polynomials that define them. In particular we look at the rank of their Hessian. We always  apply a linear transformation which transforms the $2$-jet in the form described by Jacobi's Theorem. 
\begin{algorithm} \label{al:singularities}
\hspace{1pt}\\
INPUT: A quartic polynomial $f$ with general coefficients and  Newton polytope from Proposition \ref{prop:minimal}, and a singular point $u$ on the surface defined by $f$. 
\hspace{1pt}\\
OUTPUT: The type of singularity at the point $u$ according to Arnol\cprime d's classification. 
\begin{enumerate}[1.]
\item Dehomogenize the polynomial $f$, so that the singular point $u$ is the origin. 
\item   Compute the rank of the Hessian matrix $H(f)(0)$. Depending on the rank perform the following steps. 
\item Rank 3: Apply a  linear transformation which transforms the $2$-jet in the form described by Jacobi's Theorem, 
\[ J_2 f \,\, = \,\, x^2 +y^2 +z^2.\]
Since the normal form $A_1$ is $2$-determined, conclude that this singularity is of type $A_1$.
\item Rank 2: Apply a linear transformation to  transform the 2-jet to the form 
\[ J_2 \, f \,\, = \,\,x^2 + y^2 .\] 
\begin{enumerate}[4.1]
\item If the part of degree $3$ contains the monomial $z^3$, by the $(k+1)$-determinacy of $A_k$,  conclude that the singularity is of type $A_2$. 
\item If $f$ contains the monomials $z^4$, $x  z^2$ or $y  z^2$, then, when 
applying the transformation (\ref{eq:transformation}),  the monomial $z^4$ is obtained. Deduce that the singularity is of type $A_3$.
\item If $f$ contains $x z^3$ or $y  z^3$, then, when 
applying the transformation (\ref{eq:transformation}), the monomial $z^6$ is obtained. The singularity is of type $A_5$.
\end{enumerate}
\item Rank 1: Transform the $2$-jet to the rank one quadric  $\, J_2 \, f = x^2$, and apply the transformation (\ref{eq:transformation}) in the proof of the Splitting Lemma. The result is
\[ 
f \,=\, x^2 + f_3(y,z) + f_4(y,z)  + f_5(y,z) + xg(x,y,z).\]
Here, each term in $g$ has degree at least $3$ and $f_i$ is a polynomial
of degree $i$ that depends only on the variables $y$ and $z$.
 Iterating again the steps in the proof of the Splitting Lemma will not change the $f_i$. 
 In fact, it will only produce monomials in the variables $y$ and $z$ of higher degree. 
 Let $f' = f_3(y,z) + f_4(y,z)  + f_5(y,z)$. Apply the linear automorphism described in \cite[Proposition 2.50]{GLS}. Look at the transformed polynomial $f_3$. 
\begin{enumerate}[5.1]
\item If  $f_3 = yz(y+z)$, thanks to \cite[Theorem 2.51]{GLS} conclude that the singularity is of type $D_4$.  
\item If $f_3 = y^2z$,  argue as in the proof of \cite[Theorem 2.51]{GLS}, using the fact that $D_k$ is ($k-1$)-determined.  
Write the 4-jet of $f'$ as follows,
with $\alpha, \beta \in \CC$ and $h \in \mathfrak{m}^2$:
\[J_4 \,f '\,\, = \,\, y^2z + f_4(y,z)\,\, = \,\,x^2 + y^2z + \alpha z^4 + \beta yz^3 + y^2 \cdot h(y,z),\]
 If $\alpha \not = 0$ then, as in the proof of \cite[Theorem 2.51]{GLS}, conclude that the type 
 is $D_5$. If $\alpha =0$, remark that the polynomials in our list  all contain 
 either $y z^3$ or $yz^4$. In the first situation, after applying the Tschirnhaus transformation described in the proof of the aforementioned theorem, conclude that the  singularity of type $D_6$:
\[
J_5 \, f' \,\,=\,\, y^2 z + \alpha' z^5.
\]
In the second situation, if $yz^4$ occurs, again after applying a Tschirnhaus transformation,
\[
J_6 \, f'\, \,= \,\,y^2 z + \alpha' z^6.
\]
From this form of the $6$-jet conclude that the singularity is of type $D_7$. 
\item If $f_3 = y^3$, use \cite[Theorem 2.53]{GLS}. In this case, the $4$-jet of $f'$ equals
\[
J_4 \, f' \,\,=\,\, y^3 + f_4(y,z).
\]
Remark that $f$ in our list of polynomials always contains the monomials $y^2z$, $yz^2$ or $yz^3$, so condition (b) of \cite[Theorem 2.53]{GLS} is satisfied. Therefore  conclude that the singularity is of type $E_6$, $E_7$ or $E_8$. In order to determine the type write $f$ as
\[
f' \,\,=\,\, y^3 + \alpha z^4 + \beta y z^3  + y^2 \cdot h(y,z),
\]
with $h \in \mathfrak{m}^2$. By arguing as in the proof of the mentioned theorem,  conclude that if $\alpha \not=0$, then the singularity is $E_6$, while if $\alpha = 0$, the singularity is $E_7$.
\end{enumerate}
\end{enumerate}
\end{algorithm}

\begin{table}
\begin{tabular}{|c|c||c|c||c|c|} 
\hline
type & tot. & type & tot. & type & tot. \\
\hline
$A_1$ & 22	& $D_4$ & 14	& $E_6$ & 22	\\
$A_2$ & 32	& $D_5$ & 26	& $E_7$ & 9	\\
$A_3$ & 127	& $D_6$ & 12	& & \\
$A_5$ & 58	& $D_7$ & 10	& & \\
\hline
\end{tabular}
\caption{\label{tab:sing}List of singularities of quartic surfaces arising from
the $115$ minimal polytopes in Proposition \ref{prop:minimal}. 
For each type, we list the total number of occurrences.}
\end{table}

\smallskip

\begin{proof}[Proof of Theorem~\ref{thm:2}]
We apply the analysis described above to each of the $115$ minimal polytopes in 
Proposition~\ref{prop:minimal}, and then to each coordinate point (\ref{eq:singularpoints})
that is singular in the base locus $V(m_1,\ldots,m_r)$.
Each singularity turns out to be a rational double point. 
Algorithm \ref{al:singularities} determines the singularity type according to Arnol\cprime d's classification in
Theorem~\ref{thm:class_and}. The results are summarized in Table \ref{tab:sing}.

We next apply Shah's characterization of stable quartic surfaces \cite{Mum77, Shah},
as presented in Section~\ref{sec:four}. Since all singularities are rational double points,
we conclude that all our $115$ quartic surfaces are stable
in the sense of Geometric Invariant Theory. 
Theorem~\ref{thm:2} now follows from Corollary~\ref{cor:stability}.
\end{proof}
		
\bibliographystyle{plain}

\end{document}